\documentclass{svjour3}
\usepackage[english]{babel}
\usepackage[margin=3cm]{geometry}
\usepackage[]{algorithm2e,subfig}
\usepackage{amsmath,amssymb,mathtools} \usepackage{cite,enumerate,float,pgfplots}
\usepackage{courier,dsfont}
\usepackage[utf8]{inputenc}

\title{Error estimation and uncertainty quantification for first time to a threshold value}
\author{ Jehanzeb H. Chaudhry \and Donald Estep \and Zachary Stevens \and Simon J. Tavener}
\date{\today}
\institute{J. H. Chaudhry \and Z. Stevens \at  The University of New Mexico, Albuquerque, NM 87131, USA. \email{jehanzeb@unm.edu, zstevens@unm.edu}
\and
S. J. Tavener \at Colorado State University,  Fort Collins, CO 80523, USA. \email{tavener@math.colostate.edu}
\and
D. Estep \at Simon Fraser University, Burnaby, BC V5A 1S6, Canada. \email{donald\_estep@sfu.ca}
}

\begin{document}
\maketitle

\begin{abstract}
Classical \emph{a posteriori} error analysis for differential equations quantifies the error in a Quantity of Interest (QoI) which is represented as a bounded linear functional of the solution. In this work we consider \emph{a posteriori} error estimates of a quantity of interest that cannot be represented in this fashion, namely the time at which a threshold is crossed for the first time. We derive two representations for such errors and use an adjoint-based \emph{a posteriori} approach to estimate unknown terms that appear in our representation. The first representation is based on linearizations using Taylor's Theorem. The second representation is obtained by implementing standard root-finding techniques. We provide several examples which demonstrate the accuracy of the methods. We then embed these error estimates within a framework to provide error bounds on a cumulative distribution function when the parameters of the differential equations are uncertain.

\keywords{ 	Initial value problems, Error bounds, Monte Carlo methods, Adjoint based error estimation, Uncertainty quantification}
\subclass{  65L05, 65L70, 65C05}
\end{abstract}

\section{Introduction} \label{sec:introduction}

There are many situations in which the purpose of a computation is to determine \emph{when} a functional of the solution to \eqref{eq:ivp} achieves a certain event, for example when a temperature or a species concentration reaches a specified level, the wave height of a tsunami crosses a threshold at a certain location, an orbiting body completes a revolution etc. In this article we perform \textit{a posteriori} analysis for the error in the computed value and computed probability distribution of the time at which a threshold value is realized for the first time in the context of ordinary differential equations (ODEs). More precisely, consider a system of first order ODEs
\begin{equation}
\label{eq:ivp}
\dot{y} = f(y,t; \theta), \; t \in (0,T],  \quad  y(0)= y_0,
\end{equation}
where $\dot{y} \equiv \frac{d y(t)}{dt}$, $f: \mathbb{R}^m \times \mathbb{R} \times \mathbb{R} \rightarrow \mathbb{R}^m $ is a Lipschitz continuous function and $\theta$ is a deterministic or random parameter. Let $S(y(t))$ be a linear functional of $y(t)$ and $Q(y)$ be the first time $t \in (0, T]$ at which a threshold $S(y(t)) = R $ is crossed. \textcolor{black}{We assume such a $t<T$ exists.} That is,

\begin{equation}
\label{eq:QoI}
Q(y) := \min\limits_{t \in (0,T]}  \text{arg}( S(y(t)) = R ).
\end{equation}
Hence, we refer to this as a non-standard QoI in the context of \textit{a posteriori} error analysis. An example of this non-standard QoI for the Lorenz system (see \S \ref{sec:lorenz_example_true_soln}) is illustrated in Figure~\ref{fig:Lorenz_true}.

Standard adjoint-based \emph{a posteriori} error analysis seeks to estimate the error in a quantity of interest (QoI) that can be expressed as a bounded functional of the solution and is widely used for a broad range of numerical methods~\cite{CET2014,AO2000,Ban,barth04,beckerrannacher,CaEsTa2008,Cao2004,CCS2017,CEG+2013,CEG+2015,CEGT13b,CET+2016,CET2015,Chaudhry17,Chaudhry2019,Chaudhry2019b,DEstep,eehj_actanum_95,ELM,estep_sinum_95,Estep2009,JCC+2015}.
In these cases, the error analysis involves computable residuals of the numerical solution, the generalized Green’s function solving an adjoint problem and variational analysis \cite{AO2000, Estep, beckerrannacher, DEstep, ELM}. This work is briefly summarized for initial value problems in \S \ref{sec:ivp_analysis}. It is usually employed within a finite element (variational) solution strategy, but can also be applied to finite difference and finite volume methods by recasting them as equivalent finite element methods~\cite{CEG+2015,dht_mathcomp_81,dd_mathcomp_86,EJL04,Logg2004,EGT2012,CET2015,CET2014,iternon}.
Nonlinear QoIs are treated by first linearizing around a computed solution e.g. see ~\cite{beckerrannacher, CaEsTa2008, CCH-2015}.

The goal of the current work is to derive accurate error estimates for the non-standard QoI given by equation \eqref{eq:QoI} that \emph{cannot} be expressed as a bounded linear functional of the solution $y$. In addition, we use the result to bound the error in an empirical distribution function for the nonstandard QoI corresponding to a stochastic parameter $\theta$. This is similar to the \textit{a posteriori} analysis for the error in CDF for standard QoIs for PDEs with random coefficients and random geometries appearing in \cite{ZebUQ,EstepUQ1,EstepUQ2}. \textcolor{black}{The situation is more complex for a stochastic differential equation when seeking to compute the expected value of a functional of the solution at an ``exit time'' $\tau$, when the solution first leaves a specified region, since the continuous solution trajectory may leave and re-enter the specified region undetected within a single time-step of the numerical integration scheme. Discussion of this problem appears in, e.g., \cite{gobet2001euler,dzougoutov2005adaptive,bouchard2017first}.}

We first perform a \emph{a priori} error analysis for the non-standard QoI given by equation \eqref{eq:QoI} assuming the initial value problem \eqref{eq:ivp} is solved using a numerical method with $O(h^p)$ convergence rate
and show that the error in the non-standard QoI converges at the same asymptotic rate. The  \emph{a posteriori} analysis for the error in the non-standard QoI appears in \S \ref{sec:estimates}. The first approach in \S \ref{sec:QoI_Taylor_series} takes advantage of linearization via Taylor series and employs auxiliary quantities of interest to obtain a formula that directly estimates the error in the QoI.  Our second approach, in \S \ref{sec:QoI_iterative} proceeds by using different root finding methods and again employs auxiliary quantities of interest. Numerical results supporting the accuracy of the error estimates for a deterministic system appear in \S \ref{sec:results}. Details of the error estimate for the CDF when $\theta$ is a random variable are provided in \S \ref{sec:error_cdf} and the bounds are computed for several examples in \S \ref{sec:uncertainty}.

\section{ \textit{A priori} analysis }

In this section, we present a general \textit{a priori}   result regarding the convergence of the approximate time that a threshold condition is met as the discretization is refined. Assume a continuous numerical solution, $Y(t)$, of order $p$ is computed to approximate the solution to the initial value problem \eqref{eq:ivp}. That is,
\begin{equation}
    \label{eq:order_num_soln}
    \left\| y(t)-Y(t) \right\|_{\mathbb{R}^m} \leq Ch^p,
\end{equation}
for all $t \in [0,T]$, for some constant $C > 0 $. Here $\| \cdot \|_{\mathbb{R}^m}$ denotes the standard Euclidean norm in $\mathbb{R}^m$ and $h$ denotes the  step-size used to compute the  numerical solution $Y(t)$.  For a given value of the threshold $R$, define $t_t$ and $t_c$ such that
\begin{equation} \label{eq:defn_true_approx_time}
S(y(t_t)) = R = S(Y(t_c)),
\end{equation}
where $t_t = Q(y)$ is the true value of the QoI~\eqref{eq:QoI}, and  $t_{c} = \min\limits_{t \in (0,T]}  \text{arg}( S(Y(t)) = R )$ is a computed approximation to the QoI.
Here, we assume
that $S$  satisfies the Lipschitz condition in $y$,
\begin{equation}
\label{eq:bound_S}
    | S (y_1(t)) - S (y_2(t)) | \leq  K \| y_1(t) - y_2(t)\|_{\mathbb{R}^m},
\end{equation}
for some constant $K > 0$.
Define the true error in the QoI,  $e_Q$, to be
\begin{equation} \label{eq:defn_error_QoI}
e_Q = t_t- t_c.
\end{equation}

\begin{theorem}[Convergence of the non-standard QoI]
\label{thm:aprioi_analysis}
Assume there is a numerical approximation to the solution of \eqref{eq:ivp} satisfying \eqref{eq:order_num_soln}, and the functional $S(y(t))$ is continuously differentiable with respect to $t$ in a neighborhood, $B$, which contains both the true QoI, $t_t$, as well as its numerical approximation, $t_c$. Further assume there exists an $M>0$ such that
\begin{equation}
\label{eq:deriv_S_wrt_t_bounded_below}
    \left| \frac{dS}{dt}(y(t)) \right| > M,
\end{equation}
for all $t \in B$. Then the error  $e_Q$  in the computed QoI, defined by \eqref{eq:defn_error_QoI},  satisfies the bound,
\begin{equation*}
    e_Q \leq \widehat{C}h^p,
\end{equation*}
for some constant $\widehat{C}$ which depends on $M,C$ and $K$.
\end{theorem}

\begin{proof}

Given the true solution $y(t)$ to \eqref{eq:ivp}, we  consider the functional $S$ as an explicit function of t, i.e.,
\begin{equation}
    S(y)= S(y(t)) = S(t).
\end{equation}
Since $S(y(t))$ is continuously differentiable in $t$, for $t \in B$, by the Inverse Function Theorem (see \cite{Rudin}) we have $t=t(S)$ for $S$ in the image of $B$, and
\begin{equation}
    \frac{dt}{dS}(S(y(t))) = \frac{1}{\frac{dS}{dt}(t(S))}.
\end{equation}
Applying the Mean-value Theorem (see \cite{Apostol}) we have, for some $\xi$ between $S(y(t_t))$ and $S(y(t_c))$,
\begin{equation}
\begin{gathered}\begin{aligned}
            t_t-t_c & = \frac{dt}{dS}(\xi) \left[S(y(t_t))-S(y(t_c))\right] = \frac{1}{\frac{dS}{dt}(t(\xi))}\left[S(y(t_t))-S(y(t_c))\right].
\end{aligned}\end{gathered}
\end{equation}
Adding and subtracting the term $S(Y(t_c))$ and recalling that $S(y(t_t)) = R = S(Y(t_c))$,
\begin{equation}
\begin{gathered}\begin{aligned}
    t_t-t_c &= \frac{1}{\frac{dS}{dt}(t(\xi))}\left[S(y(t_t))-S(y(t_c)) + S(Y(t_c)) - S(Y(t_c))\right], \\
            &= \frac{1}{\frac{dS}{dt}(t(\xi))}\left[ S(Y(t_c)) -S(y(t_c)) \right].
\end{aligned}\end{gathered}
\end{equation}
Taking norms (absolute values for scalars), and using \eqref{eq:deriv_S_wrt_t_bounded_below} and \eqref{eq:bound_S},
\begin{equation} \label{eq:root_error_bound}
\begin{gathered}\begin{aligned}
    \left|t_t-t_c \right|   &= \left|\frac{1}{\frac{dS}{dt}(t(\xi))}\right| \left|  S(Y(t_c)) -S(y(t_c))\right| \leq \frac{1}{M} K \left\|y(t_c)-Y(t_c) \right\|_{\mathbb{R}^m} \leq \frac{1}{M} K Ch^p.
\end{aligned}\end{gathered}
\end{equation}
Defining $\widehat{C} := \frac{K \, C}{M}   $ gives the desired result.
\end{proof}

\begin{remark}
\label{rem:err_apriori}
Theorem \ref{thm:aprioi_analysis} and \eqref{eq:root_error_bound} implies that the approximate QoI $t_c$ converges to $t_t$ with at least the same rate as the numerical solution $Y(t)$.
Further, the conclusion that the constant in the final bound \eqref{eq:root_error_bound} is \emph{inversely} proportional to the lower bound on the absolute value of the derivative accords with the intuition developed through standard root-finding techniques such as Newton's method.
\end{remark}

\subsection{Example: Lorenz System}
\label{sec:lorenz_example_true_soln}
To illustrate
the convergence results in Theorem \ref{thm:aprioi_analysis}, we  consider the Lorenz system,
\begin{equation}\label{eq:Lorenz}
\left.
\begin{gathered}\begin{aligned}
    \dot{y}_1 &= \sigma(y_2 - y_1), \\
    \dot{y}_2 &= r y_1 - y_2 - y_1 y_3, \\
    \dot{y}_3 &= y_1 y_2 - b y_3,
\end{aligned}\end{gathered}
\; \right \} \;
t \in (0, 3] \qquad \hbox{with} \qquad
\left \{
\begin{gathered}\begin{aligned}
    y_1(0)&=1, \\
    y_2(0) &= 0, \\
    y_3(0) &= 24,
\end{aligned}\end{gathered}
\; \right. \;
\end{equation}
and set $\sigma =10, r=28, \text{ and } b= \frac{8}{3}$ (see \S \ref{sec:UQLorenz} for more details of this example). We define the functional $S(y(t))=y_1(t)$ and set the threshold value $R=-10$. Figure \ref{fig:Lorenz_true} illustrates an accurate reference solution as well as the threshold value and the QoI. Figure \ref{fig:Lorenz_converge} shows the convergence rates for the error in the solution and the error in the non-standard QoI when using the cG(1) method for computing the numerical solution. The cG(1) method (see \S \ref{sec:integ_schemes} for details) has second order accuracy and this convergence rate is observed both for the solution and the non-standard QoI.

\begin{figure}[h!]
    \centering
    \subfloat[Reference solution and QoI for the Lorenz system \eqref{eq:Lorenz}.]{\includegraphics[width=7cm]{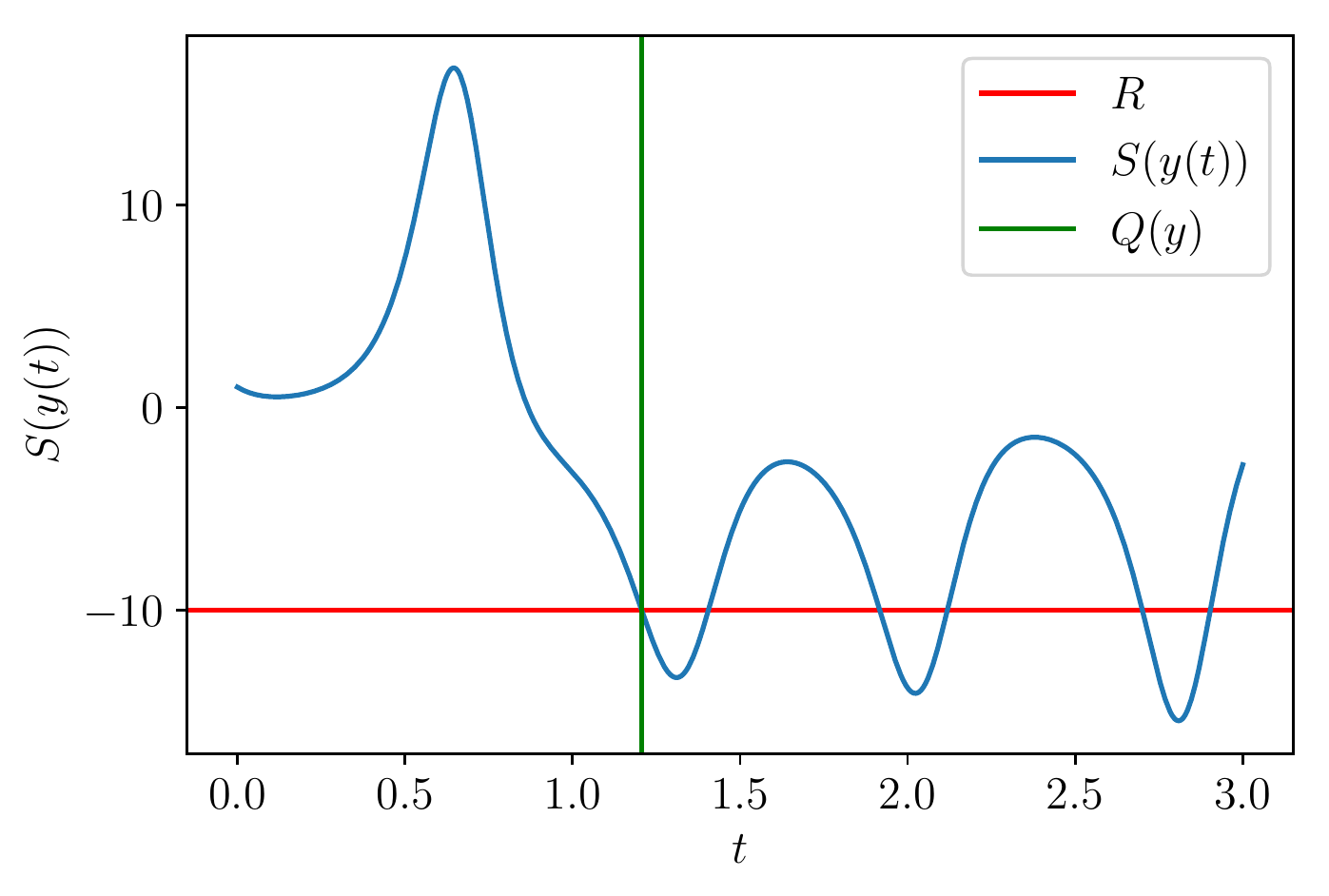}\label{fig:Lorenz_true}}
    \hfill
    \subfloat[Converge rates of the error in the solution and the error in the QoI. The numerical solution $Y$ and QoI $t_c$, are computed using the cG(1) method.]{\includegraphics[width=7cm]{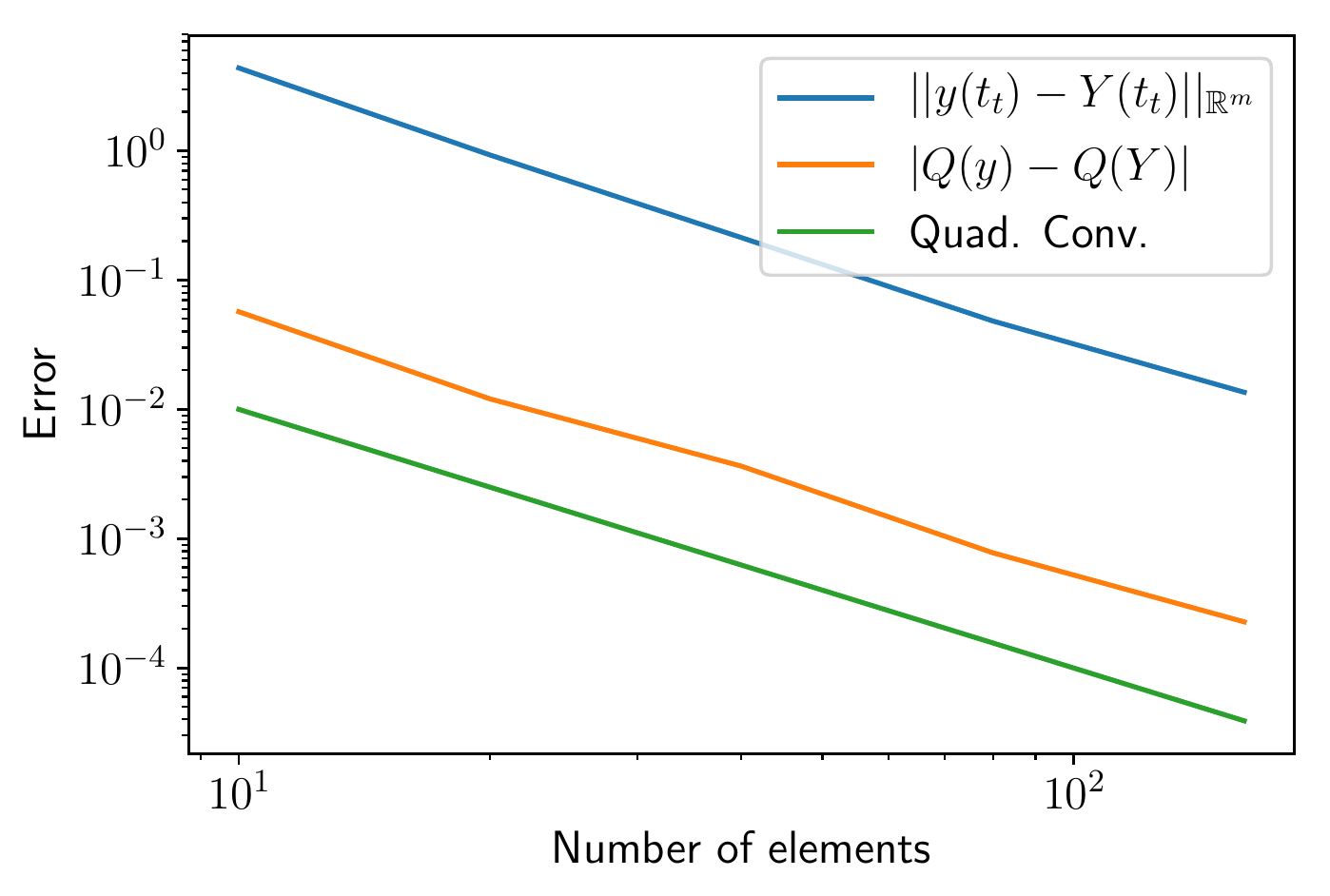}
    \label{fig:Lorenz_converge}}

    \caption{ }

\end{figure}

\section{\emph{A posteriori} error analysis}
\label{sec:estimates}

The aim  is to  derive an accurate \textit{a posteriori} error estimate $\eta \approx e_Q$.
The accuracy of the error estimate is quantified by the  effectivity ratio,
\begin{equation} \label{eq:defn_effectivity}
\rho_{\rm eff} = \frac{\eta}{e_Q}.
\end{equation}
An effectivity ratio close to one indicates an accurate error estimate. We let $\epsilon$ denote the error in the solution to \eqref{eq:ivp}, i.e.,
\begin{equation} \label{eq:defn_error_solution}
\epsilon(t) = y(t)-Y(t).
\end{equation}

\subsection{Integration schemes}
\label{sec:integ_schemes}

For simplicity, we consider a continuous FEM approximation $Y(t)$, $t \in [0,T]$, with approximate functional $S(Y(t))$ as illustrated in Figure \ref{fig:canonical_example}.
For each problem the linear functional $S(y(t))$ and the value of $R$ are specified, and the problems are solved using two different numerical schemes: (i) a variational cG(1) finite element scheme using 40 equally-sized elements and high-order Gaussian quadrature, and (ii) a Crank-Nicolson finite difference scheme with 21 equally-spaced nodes. However, we stress that the analysis can be extended to a wide variety of numerical methods \textcolor{black}{for which equivalence} to a finite element method can be established, as discussed in \S \ref{sec:introduction}.

Given  the partition $\mathcal{\tau} = \{0=t_0,t_1,...,t_N=T \}$ define the space,
\begin{equation*}
  \mathcal{V}^q  = \{ w \in C^0([0,T];\mathbb{R}^m) : w|_{I_n} \in \mathcal{P}^q(I_n), 1 \leq n \leq N \},
\end{equation*}
where $\mathcal{P}^q(I_n)$ is the space of all polynomials of degree $q$ or less on $I_n := [t_n, t_{n+1}]$.
The continuous Galerkin finite element method of order $q+1$, denoted cG($q$), for solving \eqref{eq:ivp}  is defined interval-wise by: Find $Y \in \mathcal{V}^q$ such that
\begin{equation}\label{cGq_itegral}
\int_{t_n}^{t_{n+1}}\dot{Y}(t) \cdot v(t) \; {\rm d}t = \int_{t_n}^{t_{n+1}} f(Y,t) \cdot v(t) dt,
  \quad \forall v \in \mathcal{P}^{q-1}(I_n),
\end{equation}
for $n=0,1,2,...,N-1$. The cG($q$) schemes are variational and hence well suited for adjoint based analysis. However, the Crank-Nicolson is also nodally equivalent to a variational scheme, see Theorem \ref{thm:Crank-Nicolson} in Appendix \ref{sec:proof_C-N}.

\begin{figure}[h!]
\centering
    \subfloat[Graph showing true functional $S(y(t))$, chosen value of R, and true value of QoI for the example in \S \ref{sec:example_linear}.]{\includegraphics[width=7cm]{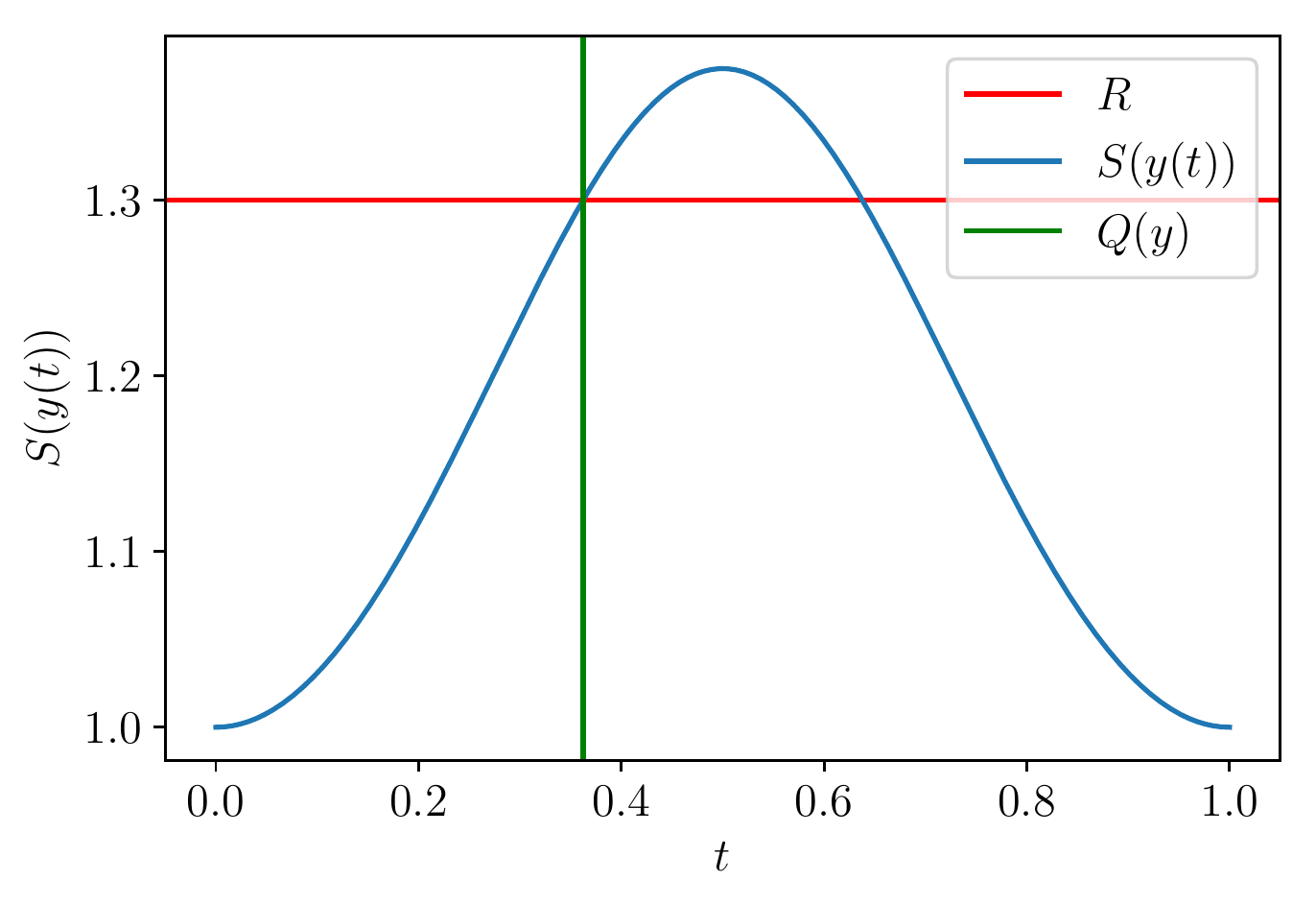}}
    \hfill
    \subfloat[Close up of true QoI and numerical QoI for the example in \S \ref{sec:example_linear}, solved using the Crank-Nicolson scheme.]{\includegraphics[width=7.46cm]{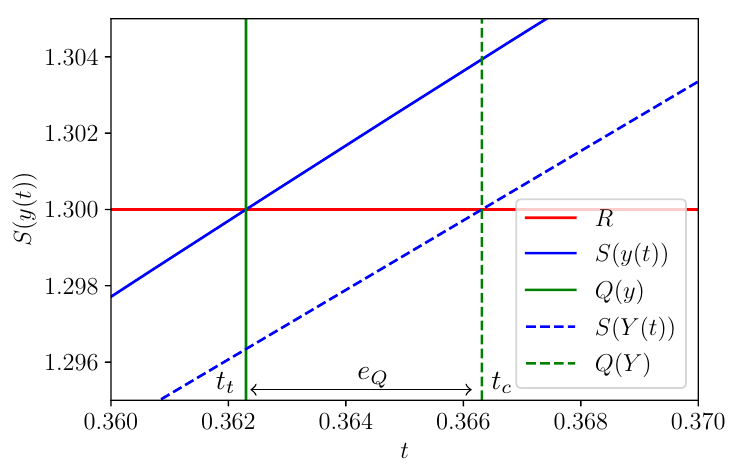}}
\caption{}
\label{fig:canonical_example}
\end{figure}

\subsection{Adjoint-based \emph{a posteriori} error analysis for standard QoIs}
\label{sec:ivp_analysis}
We derive error estimates for the nonstandard QoI in terms of expressions involving errors in linear functionals of the numerical solution. This section presents a standard \textit{a posteriori} error estimate for a linear functional of a solution. Let $(\cdot,\cdot)$ denote the inner-product pairing in $\mathbb{R}^{\textcolor{black}{m}}$.

\begin{theorem}[Adjoint-based \emph{a posteriori} error analysis for IVPs]
\label{thm:basic_error}

Given a finite element solution Y(t) of \eqref{eq:ivp} and $\psi \in \mathbb{R}^{\textcolor{black}{m}}$, the error  $(\psi,\epsilon(\hat{t}))$ at $\hat{t} \in (0,T]$ is represented as

\begin{equation} \label{eq:ivp_error_representation}
(\psi,\epsilon(\hat{t}))
= (\psi,y(\hat{t}))-(\psi,Y(\hat{t}))
= \int_{0}^{\hat{t} \ } (\phi , [ f(Y,t) - \dot{Y} ]) \ dt,
\end{equation}
where $\phi$ is the solution to the adjoint equation
\begin{equation} \label{eq:ivp_adjoint}
\begin{cases}
   -\dot{\phi} = \overline{f_{y,Y}(t)}^T\phi, \;  t \in \textcolor{black}{[0,\hat{t})},\\
    \phi(\hat{t}) = \psi,
\end{cases}
\end{equation}
with
\begin{equation}\label{dfdz}
    \overline{f_{y,Y}(t)} = \int_0^1 \frac{df}{dz}(z,t) {\rm d}s
\end{equation}
and $z=sy+(1-s)Y$,
\end{theorem}

\begin{proof}
The proof is standard see \cite{eehj_actanum_95}.
\end{proof}

\textcolor{black}{
Note that the adjoint equation \eqref{eq:ivp_adjoint} is solved backward in time from $\hat{t}$ to $0$.
}

\subsection{\textit{A posteriori analysis} for  the non-standard QoI based on Taylor series}
\label{sec:QoI_Taylor_series}

We denote the error in the non-standard QoI as $e_Q = t_t - t_c$.

\begin{theorem}
\label{thm:error_1st_via_Taylor}
For an approximate solution $Y(t)$ to \eqref{eq:ivp} and a bounded linear functional $S(y(t))$ on  $(H^1((0,T]))^{\textcolor{black}{m}}$, \textcolor{black}{if the function $f(y,t)$ is continuously differentiable in $t$, then} the error in the QoI \eqref{eq:QoI} is given by
\begin{equation}\label{eq:Taylor_error}
\textcolor{black}{
e_Q
=
\frac{S(Y(t_c))-S(y(t_c))-\mathcal{R}_1(t_c,t_t)}
             {\nabla_yS(Y(t_c))\cdot f(Y(t_c),t_c) + \nabla_y [\nabla_yS(Y(t_c))
                  \cdot f(Y(t_c),t_c)] \cdot (y(t_c)-Y(t_c)) + \mathcal{R}_2(Y(t_c))} \,,
                  }
\end{equation}
where the remainder terms $\mathcal{R}_1(t_t,t_c)$ and $\mathcal{R}_2(t_c)$ satisfy

\begin{equation}\label{eq:Taylor_remain1}
    \mathcal{R}_1(t_t,t_c) = \frac{1}{2}\frac{d^2S}{dt^2}(y(\xi))(t_{t}-t_c)^2,
    \end{equation}
for some $\xi$ between $t_t$ and $t_c$
and
\begin{equation*}\label{eq:Taylor_remain2}
\mathcal{R}_2(Y(t_c)) = ||y(t_c)-Y(t_c)||\mathcal{H}_2(Y(t_c)), \; \hbox{for $\mathcal{H}_2$ with} \;
  \lim_{Y(t_c)\rightarrow y(t_c)}\mathcal{H}_2(Y(t_c)) = 0,
\end{equation*}
and $|| \cdot ||$ denotes the Euclidean norm on $\mathbb{R}^{\textcolor{black}{m}}$.
\end{theorem}
\begin{proof}

From the definition of the  functional $S(y(t))$ and $R$,
\begin{equation}\label{equiv}
    S(Y(t_c)) = R = S(y(t_{t})).
\end{equation}
Expanding $S(y(t_{t}))$ using Taylor's Theorem with remainder centered at $t_c$ (e.g. see \cite{Apostol}) in \eqref{equiv},
\begin{equation}\label{firstapprox}
    S(Y(t_c)) = S(y(t_c)) + \frac{dS}{dt}(y(t_c))(t_{t}-t_c) + \mathcal{R}_1(t_c,t_t),
\end{equation}
Applying the chain-rule to the derivative in \eqref{firstapprox} and using \eqref{eq:ivp} gives
\begin{equation}\label{chain_rule}
\begin{aligned}
S(Y(t_c)) = S(y(t_c))
+ \textcolor{black}{ \left[ \nabla_yS(y(t_c))\cdot f(y(t_c),t_c) \right] } (t_{t}-t_c)
+ \mathcal{R}_1(t_c,t_t).
\end{aligned}
\end{equation}
Adding and subtracting the term $ \nabla_yS(Y(t_c))\cdot f(Y(t_c),t_c)$ inside the square brackets gives
\begin{equation}\label{halfway}
\begin{aligned}
&S(Y(t_c))= S(y(t_c)) \\
&+\textcolor{black}{ \left[ \nabla_yS(Y(t_c))\cdot f(Y(t_c),t_c) + \left( \nabla_yS(y(t_c))\cdot f(y(t_c),t_c) - \nabla_yS(Y(t_c))\cdot f(Y(t_c),t_c) \right) \right] } (t_{t}-t_c) \\
&+ \mathcal{R}_1(t_c,t_t).
\end{aligned}
\end{equation}
Using the multi-variable Taylor's Theorem with remainder centered at $Y$ (e.g. see \cite{Apostol2}) gives
\begin{equation}\label{secondapprox}
    \nabla_yS(y(t_c))\cdot f(y(t_c),t_c) - \nabla_yS(Y(t_c))\cdot f(Y(t_c),t_c) = \nabla_y[\nabla_yS(Y(t_c))\cdot f(Y(t_c),t_c)]\cdot(y(t_c)-Y(t_c)) +\mathcal{R}_2(Y(t_c)),
\end{equation}
where the remainder is of the form
\begin{equation}
\label{eq:rem_r_2}
    \mathcal{R}_2(Y(t_c)) = \frac{1}{2}(Y(t_c)-y(t_c))^\top \textbf{H}_y(\nabla_yS(\xi)\cdot f(\xi,t_c))(Y(t_c)-y(t_c)),
\end{equation}
for some $\xi$ between $y(t_c)$ and $Y(t_c)$, and
where $\textbf{H}_y$ is the Hessian
\begin{equation*}
    (\textbf{H}_y)_{i,j} = \frac{\partial^2}{\partial y_i \partial y_j}.
\end{equation*}
Substituting \eqref{secondapprox} in to \eqref{halfway} and rearranging to isolate the error of the QoI, results in
\begin{equation}\label{direct}
\begin{aligned}
&(t_{t}-t_c) = \\
&\textcolor{black}{\frac{S(Y(t_c))-S(y(t_c))-\mathcal{R}_1(t_c,t_t)}{ \nabla_y S(Y(t_c))\cdot f(Y(t_c),t_c) + \nabla_y[\nabla_y S(Y(t_c))\cdot f(Y(t_c),t_c)]\cdot(y(t_c)-Y(t_c))+\mathcal{R}_2(Y(t_c))}}.
\end{aligned}
\end{equation}

\end{proof}

\begin{corollary}\textcolor{black}{
For functionals $S(y(t),t)$ that are explicitly dependent on $t$,}
\begin{equation}\label{eq:Taylor_error_simp1}\textcolor{black}{
e_Q = \frac{S(Y(t_c),t_c)-S(y(t_c),t_c)-\mathcal{R}_1(t_c,t_t)}
                { \dfrac{\partial S}{\partial t}(y(t_c),t_c)  + \nabla_y S(Y(t_c),t_c)\cdot f(Y(t_c),t_c) + \nabla_y [\nabla_y S(Y(t_c),t_c) \cdot f(Y(t_c),t_c)]
                   \cdot (y(t_c)-Y(t_c)) + \mathcal{R}_2(Y(t_c)))} } \,.
\end{equation}
\textcolor{black}{Where the partial derivative of $S$ with respect to $t$ appears from the chain-rule applied to \eqref{firstapprox}.}
\end{corollary}

\textcolor{black}{
\begin{proof}
If $S$ depends explicitly on $t$, then \eqref{chain_rule} becomes
\begin{align*}
S(Y(t_c),t_c) = S(y(t_c),t_c)
+ \left[ \nabla_yS(y(t_c),t_c)\cdot f(y(t_c),t_c) + \dfrac{\partial S}{\partial t}(y(t_c),t_c) \right] (t_{t}-t_c)
+ \mathcal{R}_1(t_c,t_t).
\end{align*}
The remainder of the proof mimics the proof of Theorem \ref{thm:error_1st_via_Taylor} retaining this additional partial derivative.
\end{proof}
}

Note that functionals \textcolor{black}{$S(y(t),t)$} that depend directly on $t$ require special treatment of the term \textcolor{black}{$\frac{\partial S}{\partial t}(y(t_c),t_c)$} in \textcolor{black}{\eqref{eq:Taylor_error_simp1}}. More precisely, one can use another application of Taylor's Theorem centered at $Y(t_c)$ in order to make this term computable.
\begin{corollary}
For functionals of the form $S(y(t)) = v \cdot y(t), \text{ for some } v \in \mathbb{R}^{\textcolor{black}{m}}$, $\nabla_y S(y(t)) = v$, and
\eqref{eq:Taylor_error}
becomes
\begin{equation}\label{eq:Taylor_error_simp2}
\begin{gathered}\begin{aligned}
e_Q &= \frac{ -v \cdot (y(t_c)-Y(t_c)) - \mathcal{R}_1(t_c,t_t) }
                 { v\cdot f(Y(t_c),t_c) + v^\top \nabla_y f(Y(t_c),t_c) \cdot (y(t_c)-Y(t_c))
                   + \mathcal{R}_2(Y(t_c))}  \\
&= \frac{ -v \cdot \epsilon(t_c) - \mathcal{R}_1(t_c,t_t) }
                { v\cdot f(Y(t_c),t_c) + v^\top \nabla_y f(Y(t_c),t_c) \cdot \epsilon(t_c)
                   + \mathcal{R}_2(Y(t_c))}
\end{aligned}\end{gathered}
\end{equation}
\end{corollary}

\bigskip
\noindent{\bf Obtaining a computable error estimate.}
Taylor's Theorem gives that the two functions $\mathcal{R}_1$ and $\mathcal{R}_2$ in equations \eqref{eq:Taylor_error} and \eqref{eq:Taylor_error_simp2} decay to zero super-linearly as $t_c \rightarrow t_t$ and $Y(t_c) \rightarrow y(t_c)$, respectively. Provided  the numerical solution $Y(t)$ is fairly accurate, $\mathcal{R}_1$ will be small compared to the other terms in \eqref{firstapprox} and $\mathcal{R}_2$ will be small compared to the terms in \eqref{secondapprox}. This leads to the first approximation of the error,
\begin{equation}\label{eq:taylor_eta_approx}
    \eta(Y) = \frac{-v \cdot \epsilon(t_c)}
                   { v\cdot f(Y(t_c),t_c) + (v^\top \nabla_y f(Y(t_c),t_c))\cdot \epsilon(t_c)} \,.
\end{equation}

\begin{remark}
\label{rem:err_est_accuracy}

Note that the functional $S$ may achieve the value $R$ at multiple times.
Assume there exists a time $\textcolor{black}{\tilde t} > t_t$ such that $S(y(\textcolor{black}{\tilde t})) = R$.
Equation~\eqref{equiv} is then valid at time $\textcolor{black}{\tilde t}$, i.e., \textcolor{black}{$S(Y(t_c)) = R = S(y(\tilde t))$} and  \eqref{eq:Taylor_error} follows with $t_t$ replaced by $\textcolor{black}{\tilde t}$ and $\xi$ replaced by \textcolor{black}{$\tilde\xi$}.
In the estimate \eqref{eq:taylor_eta_approx} we approximate the term $\mathcal{R}_1(t_c,\cdot)$ by zero.
If the numerical solution is sufficiently accurate, then $\vert t_t - t_c \vert < \vert \textcolor{black}{\tilde t} - t_c \vert$ and $0 \approx \mathcal{R}_1(t_c,t_t) \ll  \mathcal{R}_1(t_c,\textcolor{black}{\tilde t})$.
However, if the numerical solution is inaccurate, we may have the reverse situation, where $\vert t_t - t_c \vert > \vert \textcolor{black}{\tilde t} - t_c \vert$,
in which case the error estimate will be inaccurate or worse,
$\mathcal{R}_1(t_c,\textcolor{black}{\tilde t}) \approx 0$
and the estimate may indicate the value of $\textcolor{black}{\tilde t} - t_c$ rather than $t_t - t_c$.
We observe this phenomenon in \S \ref{sec:example_harmonic_interval} and is illustrated by Table \ref{tab:shiftosc_CN} and Figure \ref{fig:shiftosc_CN}.
\end{remark}

The estimate \eqref{eq:taylor_eta_approx} contains two terms that are linear functionals of the error at time $t_c$. These can both be approximated by the standard techniques in \S \ref{sec:ivp_analysis} as is discussed next.

\medskip
\noindent{\bf First adjoint problem}
In order to estimate $-v  \cdot \epsilon(t_c)$, we solve \eqref{eq:ivp_adjoint} with adjoint data $\psi = \psi_1 = -v$ and $\hat{t} = t_c$, then substitute the solution $\phi_1$ in \eqref{eq:ivp_error_representation} to provide the estimate
\begin{equation}\label{eq:taylor_numerator_estimate}
    \mathcal{E}_1(Y,\phi_1) \approx  \psi_1 \cdot \epsilon(t_c)   = -v  \cdot \epsilon(t_c).
\end{equation}

\medskip
\noindent{\bf Second adjoint problem}
In order to estimate $v^T \nabla_y f(Y(t_c),t_c) \cdot \epsilon(t_c)$, we solve \eqref{eq:ivp_adjoint} with adjoint data $\psi = \psi_2 = v^T\nabla_y f(Y(t_c),t_c)$ and $\hat{t} = t_c$, then substitute the solution $\phi_2$ in \eqref{eq:ivp_error_representation} to provide the estimate
\begin{equation}\label{eq:taylor_denominator_estimate}
    \mathcal{E}_2(Y,\phi_2) \approx \psi_2 \cdot \epsilon(t_c)   =  v^\top \nabla_y f(Y(t_c),t_c) \cdot \epsilon(t_c).
\end{equation}

\medskip

\noindent{\bf Computable error based on Taylor series and adjoint techniques.}
For an approximate solution $Y(t)$ to \eqref{eq:ivp} and a linear functional $S(Y(t))=v\cdot Y(t)$, a computable estimate of the error in the QoI \eqref{eq:QoI}  is obtained by
substituting \eqref{eq:taylor_numerator_estimate} and \eqref{eq:taylor_denominator_estimate} in   \eqref{eq:taylor_eta_approx},
\begin{equation}\label{eq:Taylor_computable}
    \eta(Y,\phi_1,\phi_2) = \frac{\mathcal{E}_1(Y,\phi_1)}{ v\cdot f(Y(t_c),t_c) + \mathcal{E}_2(Y,\phi_2)} \,.
\end{equation}

\subsection{Error in non-standard QoI based on iterative techniques}
\label{sec:QoI_iterative}

Given an approximate solution $Y(t)$ to \eqref{eq:ivp} with numerical QoI $t_c$, define $g(t)$ as
\begin{equation}
\begin{gathered}\begin{aligned}
g(t) &= S(y(t)) - R, \\
     &= S(Y(t)) + \big( S(y(t))-S(Y(t)) \big) - R, \\
\end{aligned}\end{gathered}
\end{equation}
so
\begin{equation}
    g(t_t) = 0.
\end{equation}
In the case where $S(t)$ is a linear functional of $y(t)$, i.e., $S(y(t)) = v \cdot y(t)$, then
\begin{equation*}
g(t) = S(Y(t)) +  v \cdot \epsilon(t) - R.
\end{equation*}
At $t=\hat{t}$ we estimate $v \cdot \epsilon(\hat{t})$ by solving  \eqref{eq:ivp_adjoint} with adjoint data $\psi = \psi_3 = v^\top$ and substituting the solution $\phi_3$ in to \eqref{eq:ivp_error_representation} to provide the estimate
\begin{equation}\label{eq:iterative_est}
    \mathcal{E}_3(Y,\phi_3;\hat{t}) \approx  v^\top \cdot(y(\hat{t})-Y(\hat{t})),
\end{equation}
hence
\begin{equation*}
    g(\hat{t}) =  S(Y(\hat{t})) + \mathcal{E}_3(Y,\phi_3; \hat{t}) - R.
\end{equation*}
We find $t^*$ such that $g(t^*) \approx 0$ via a standard root finding procedure, then
\begin{equation} \label{eq:iterative_estimate}
\eta(Y) = t^*-t_c.
\end{equation}
There are many options for root finding methods for  computing $\eta$. In this article, we use two of the basic root finding methods: the secant method and the inverse quadratic method.

\subsubsection{Error estimate based on the secant method}
\label{sec:iterative_error_secant}
Given initial values $x_0, \, x_1$, the method is defined by the recurrence
\begin{equation}\label{secant}
    x_{n} =\frac{x_{n-2}*g(x_{n-1})-x_{n-1}*g(x_{n-2})}{g(x_{n-1})-g(x_{n-2})} \quad n=2,3,\dots
\end{equation}
(See \cite{Gautschi}). For the initial guesses the examples presented choose $x_0 < t_c < x_1$. These choices are made precise in the numerical examples in \S \ref{sec:results}.

\subsubsection{Error estimate based on inverse quadratic interpolation}
\label{sec:iterative_error_inverse_quadratic}

Given initial values $x_0, \, x_1, \, x_2$, the method is defined by the recurrence
\begin{equation}\label{invquad}
\begin{gathered}\begin{aligned}
&x_{n} =
 \frac{x_{n-3} \ g_{n-2} \ g_{n-1}}{(g_{n-3}-g_{n-2})(g_{n-3}-g_{n-1})} +
 \frac{x_{n-2} \ g_{n-3} \ g_{n-1}}{(g_{n-2}-g_{n-3})(g_{n-2}-g_{n-1})} +
 \frac{x_{n-1} \ g_{n-2} \ g_{n-3}}{(g_{n-1}-g_{n-2})(g_{n-1}-g_{n-3})}\,. \\
&\qquad n=3,4,\dots.
\end{aligned}\end{gathered}
\end{equation}
(See \cite{Epper}). The choice of the initial guesses is made precise in the numerical examples in \S \ref{sec:results}.

\subsection{Comparison of the two error estimation methods}

The method based on Taylor series always requires fewer adjoint problems to be solved than using one of the iterative methods. However, the estimate \eqref{eq:taylor_eta_approx} neglects certain terms compared to the error representation \eqref{eq:Taylor_error_simp2}. If any of the neglected terms are large, the error estimate may be inaccurate even though an accurate numerical solution is used. The iterative methods only rely on the initial guesses and point-wise error computation, which is computed accurately. The initial guesses defined in Section \ref{sec:results} bracket the computed QoI, and provided the computed solution is sufficiently accurate and the initial bracket contains only a single value $t$ such that $S(y(t)=R$, the iterative methods will be accurate.
Numerical comparisons of the two methods, as well as limitations of both are discussed throughout Section \ref{sec:results}.

\begin{color}{black}
\subsection{Error in a cumulative density function}
\label{sec:error_cdf}
	
If the differential equation \eqref{eq:ivp} depends on a random parameter $\theta$, then the solution $y(t;\theta)$ and the QoI, $Q(y;\theta)$, are random variables. As a random variable, $Q(y;\theta)$ has a corresponding cumulative distribution function (CDF),
\begin{equation*}
	F(t) = P(\{ \theta: Q(y;\theta)\leq t \}) = P(Q\leq t).
\end{equation*}
An approximation to the CDF  is computed using  the Monte Carlo method with a finite number of numerically computed sample values
$\{ \hat{Q}(Y^{[n]},\theta^{[n]} ) = \hat{Q}^{[n]}\}_{n=1}^N$,
\begin{equation}\label{approxCDF}
\hat{F}_N(t) = \frac{1}{N}\sum_{n=1}^N {\mathds{1}( \hat{Q}^{[n]}\leq t) },
\end{equation}
where $\mathds{1}$ is the indicator function. A nominal sample distribution is computed using exact values of the QoI,
\begin{equation}\label{nomCDF}
F_N(t) = \frac{1}{N}\sum_{n=1}^N {\mathds{1}( Q^{[n]}\leq t)}.
\end{equation}
	
An estimate of the error in an approximate distribution of the QoI \eqref{eq:QoI} is computed for two examples in \S \ref{sec:uncertainty}. The estimate takes into account error contributions due to finite sampling and errors arising from the discretization of the ODE.
The expressions \eqref{approxCDF} and \eqref{nomCDF}  decompose the error in to sampling and {discretization} contributions,
\begin{equation*}
F(t)-{\hat{F}_N(t)} = (F(t)-F_N(t)) + (F_N(t)-\hat{F}_N(t)).
\end{equation*}
This decomposition is used to derive the following error bound.
	
\begin{theorem}
	\label{thm:cdf_bound}
	For $0 < \varepsilon < 1$,

	\begin{align}\label{CDFerr}
	\left|F(t)-\hat{F}_N(t) \right|
	&\leq\left(\frac{\hat{F}_N(t)\left(1-\hat{F}_N(t) \right)}{N\varepsilon} \right)^{1/2} + \left( \frac{1}{N} + \frac{1}{N\varepsilon^{1/2}} \right) \left|\sum_{n=1}^N\left( \mathds{1}(\hat{Q}^{[n]}-\left|e_Q^{[n]} \right| \leq t \leq \hat{Q}^{[n]}+\left|e_Q^{[n]} \right| ) \right) \right| \nonumber \\
	& \ \ \ +\frac{2}{\left(2N\varepsilon\right)^{3/4}}
	\end{align}
	with probability greater than or equal to $1-2\varepsilon+\varepsilon^2$, where $e_Q^{[n]} = Q^{[n]}-\hat{Q}^{[n]}$ is the error in a numerically computed sample of the QoI.
\end{theorem}
	
	\begin{proof}
		We decompose the error as
		\begin{equation}
        \label{eqcdf_err_decomp_a}
		\left|F(t)-\hat{F}_N(t) \right| \leq \left|F(t) - F_N(t) \right|+ \left|F_N(t)-\hat{F}_N(t)  \right| = I+II.
		\end{equation}
		Focusing on the term $II = \left|F_N(t)-\hat{F}_N(t)  \right| = \left|\hat{F}_N(t)-F_N(t)  \right|$,
		\begin{align}
        \label{bound2}
		II &= \left|\frac{1}{N}\sum_{n=1}^N\left( \mathds{1}(\hat{Q}^{[n]}\leq t) - \mathds{1}(Q^{[n]}\leq t)  \right) \right| =\left|\frac{1}{N}\sum_{n=1}^N\left( \mathds{1}(\hat{Q}^{[n]}\leq t) - \mathds{1}(\hat{Q}^{[n]}+e_Q^{[n]}\leq t)  \right) \right|, \nonumber \\
		&=\left|\frac{1}{N}\sum_{\substack{n=1 \\ e_Q^{[n]}\leq 0 }}^N\left( \mathds{1}(\hat{Q}^{[n]}-\left|e_Q^{[n]} \right| \leq t \leq \hat{Q}^{[n]} ) \right) + \frac{1}{N}\sum_{\substack{n=1 \\ e_Q^{[n]}> 0 }}^N\left( \mathds{1}(\hat{Q}^{[n]} \leq t \leq \hat{Q}^{[n]} +\left|e_Q^{[n]} \right| ) \right) \right|, \nonumber \\
		&\leq \left|\frac{1}{N}\sum_{n=1}^N\left( \mathds{1}(\hat{Q}^{[n]}-\left|e_Q^{[n]} \right| \leq t \leq \hat{Q}^{[n]} ) \right) + \frac{1}{N}\sum_{n=1}^N\left( \mathds{1}(\hat{Q}^{[n]} \leq t \leq \hat{Q}^{[n]} +\left|e_Q^{[n]} \right| ) \right) \right|, \nonumber \\
		&= \left|\frac{1}{N}\sum_{n=1}^N\left( \mathds{1}(\hat{Q}^{[n]}-\left|e_Q^{[n]} \right| \leq t \leq \hat{Q}^{[n]}+\left|e_Q^{[n]} \right| ) \right) \right|,
		\end{align}
										
		Now consider the term $I=\left|F(t) - F_N(t) \right|$. We start with the Chebyshev Inequality:
		\begin{align*}
		P\left(\left|F(t)-F_N(t) \right|\geq ks \right) \leq \frac{1}{k^2}
		\end{align*}
		for any real number $k$, where $s^2$ is the variance of $F_N$ given by \cite{EstepUQ1,Serfling},
		\begin{equation*}
		s^2 = \frac{F(t)\left(1-F(t)\right)}{N}.
		\end{equation*}
		Choosing $\varepsilon=\frac{1}{k^2}$ leads to
		\begin{align}\label{bound_I}
		I=\left|F(t)-F_N(t) \right| \leq \left(\frac{F(t)\left(1-F(t) \right)}{N\varepsilon} \right)^{1/2},
		\end{align}
		with a probability greater than $1-\varepsilon$.
		Now,
		\begin{align}
        \label{eq:use_eq_cdf}
		F(t)\left(1-F(t)\right) = F_N(t)\left(1-F_N(t)\right) + \left(F(t)-F_N(t)\right)\left(1-F(t)-F_N(t)\right).
		\end{align}
		Taking absolute values in \eqref{eq:use_eq_cdf}, dividing by $N\varepsilon$, taking the square root, and using  $\sqrt{a+b}\leq \sqrt{a}+ \sqrt{b}$ for any $a,b\geq 0$,
								                \begin{align}
        \left| \frac{F(t)\left(1-F(t)\right)}{N\varepsilon}  \right|^{1/2} &\leq \left| \frac{F_N(t)\left(1-F_N(t)\right)}{N\varepsilon} \right| ^{1/2} +\left|\frac{\left(F(t)-F_N(t)\right)\left(1-F(t)-F_N(t)\right)}{N\varepsilon} \right| ^{1/2}.\label{this}
        \end{align}
        												        
                           Multiplying and dividing the second term on the right-hand side of \eqref{this} by $\sqrt{2}\delta$ and using the fact that $ab\leq \frac{a^2}{2}+\frac{b^2}{2} $,
                  \begin{align*}
        \left|\frac{\left(F(t)-F_N(t)\right)\left(1-F(t)-F_N(t)\right)}{N\varepsilon} \right| ^{1/2} &\leq \left|\delta^2\left(F(t)-F_N(t) \right)^2 + \frac{\left(1-F(t)-F_N(t)\right)^{2}}{4\delta^2N^2\varepsilon^2}  \right|^{1/2} \\
        &\leq \delta \left| F(t)-F_N(t) \right| + \frac{1}{2 \delta N \varepsilon},
        \end{align*}
		where we obtain the final line by observing that $\left(1-F(t)-F_N(t)\right)^2 \leq 1 $.
                Substituting back into \eqref{this} and combining with \eqref{bound_I}, 		
		\begin{align}\label{nearly_delta_bound}
		I &\leq \left( \frac{F_N(t)\left(1-F_N(t)\right)}{N\varepsilon} \right) ^{1/2}  +\delta\left| F(t)-F_N(t) \right| + \frac{1}{2 \delta N \varepsilon}.
		\end{align}
				From \cite{Serfling}, for any $\varepsilon>0$ we have with a probability greater than $1-\varepsilon$,
		\begin{align}\label{log_bound}
		I \leq \left(\dfrac{\log(\varepsilon^{-1})}{2N}\right)^{1/2} \leq \left(\dfrac{1}{2N\varepsilon}\right)^{1/2},
		\end{align}
										where we also used that $\log(x)\leq x$ for all $x>0$. Substituting this into the right-hand side of \eqref{nearly_delta_bound},
		\begin{align}\label{nearly_alt}
		I &\leq \left( \frac{F_N(t)\left(1-F_N(t)\right)}{N\varepsilon} \right) ^{1/2}  +\delta\left(\dfrac{1}{2N\varepsilon}\right)^{1/2} + \frac{1}{2\delta N \varepsilon}.
		\end{align}
		Consider the function
		\begin{equation*}
		        D(\delta)=\frac{\delta}{(2N\varepsilon)^{1/2}}+\frac{1}{\delta (2N\varepsilon)},
		\end{equation*}
		        Elementary calculus shows that the minimum of $D(\delta)$, for $\delta > 0$, occurs at $\delta = \left(\frac{1}{2N\varepsilon}\right)^{1/4}$.
  								        With this choice of $\delta$, \eqref{nearly_alt} becomes
		\begin{align}
		I &\leq \left( \frac{F_N(t)\left(1-F_N(t)\right)}{N\varepsilon} \right) ^{1/2}  + \frac{2}{\left(2N\varepsilon \right)^{3/4}}.
        \label{eq:ineq_diff_true_nom_cdf}
		\end{align}
		The numerator of the first term  in \eqref{eq:ineq_diff_true_nom_cdf} is expanded as
		\begin{align}
        \label{eq:help_eqn_1000}
		\left| F_N(t)\left(1-F_N(t)\right)\right| &= \left|\hat{F}_N(t)\left(1-\hat{F}_N(t) \right) + \left(F_N(t)-\hat{F}_N(t) \right)\left(1-F_N(t)-\hat{F}_N(t) \right)\right| \nonumber \\
		&\leq \left|\hat{F}_N(t)\left(1-\hat{F}_N(t) \right)\right| + \left| \left(F_N(t)-\hat{F}_N(t) \right)\left(1-F_N(t)-\hat{F}_N(t) \right)\right|.
		\end{align}
        Using  $\left|1-F_N(t)-\hat{F}_N(t) \right| \leq 1 $ in \eqref{eq:help_eqn_1000} together with \eqref{bound2} and \eqref{eq:ineq_diff_true_nom_cdf},

		\begin{align}\label{bound1}
		I &\leq \left(\frac{\hat{F}_N(t)\left(1-\hat{F}_N(t) \right)}{N\varepsilon} \right)^{1/2} + \frac{1}{N\varepsilon^{1/2}}\left(\left|\sum_{n=1}^N\left( \mathds{1}(\hat{Q}^{[n]}-\left|e_Q^{[n]} \right| \leq t \leq \hat{Q}^{[n]}+\left|e_Q^{[n]} \right| ) \right) \right| \right)^{1/2}+\frac{2}{\left(2N\varepsilon\right)^{3/4}}\nonumber ,\\
        &\leq \left(\frac{\hat{F}_N(t)\left(1-\hat{F}_N(t) \right)}{N\varepsilon} \right)^{1/2} + \frac{1}{N\varepsilon^{1/2}}\left(\left|\sum_{n=1}^N\left( \mathds{1}(\hat{Q}^{[n]}-\left|e_Q^{[n]} \right| \leq t \leq \hat{Q}^{[n]}+\left|e_Q^{[n]} \right| ) \right) \right| \right)+\frac{2}{\left(2N\varepsilon\right)^{3/4}},
		\end{align}
		where we also used $\sqrt{x} \leq x$ if $x = 0$ or $x \geq 1$.
        		Since \eqref{bound1} relies on both \eqref{bound_I} and \eqref{log_bound}, this bound occurs with a probability of at least $(1-\varepsilon)^2 = 1-2\varepsilon+\varepsilon^2$. Combining
        \eqref{bound2} and \eqref{bound1} with \eqref{eqcdf_err_decomp_a} completes the proof.
		
	\end{proof}
	
\end{color}
The estimate \eqref{eq:Taylor_computable} is used to approximate $\eta^{[n]} \approx e_Q^{[n]}$. The first term on the right-hand side of the bound \eqref{CDFerr} quantifies the error contribution from finite sampling, while the second term represents error due to discretization.

\section{Numerical examples}
\label{sec:results}
This section considers a wide range of types of linear and nonlinear ODEs in order to explore the accuracy of the estimates.

Since the Crank-Nicolson finite difference scheme is nodally equivalent to the cG(1) finite element method with a trapezoidal rule quadrature, given $t_i < t_c < t_{i+1}$,  the numerical QoI may be computed by using linear interpolation as,
\begin{equation*}\label{tnum}
    t_c = \frac{R(t_i-t_{i+1})}{Y(t_i)-Y(t_{i+1})} - \frac{t_iY(t_i)-t_{i+1}Y(t_{i+1})}{Y(t_i)-Y(t_{i+1})}.
\end{equation*}

When implementing the secant method \eqref{secant}, the two grid-points closest to the QoI are used as initial guesses:
\begin{equation}
    x_0 = t_L \text{ and } x_1=t_R,
\end{equation}
where $t_L < t_c <t_R$,  with no other grid-points in between.
For the inverse quadratic interpolation scheme \eqref{invquad}, the initial guesses are the two closest grid-points to the left of the QoI and one to the right:
\begin{equation}
    x_0 = t_{LL}, \;  x_1=t_L \text{ and } x_2=t_R,
\end{equation}
where $t_{LL} < t_L < t_c <t_R$, with no other grid-points in between.
For most examples the adjoint solutions, needed for the estimates~\eqref{eq:taylor_numerator_estimate}, \eqref{eq:taylor_denominator_estimate} and \eqref{eq:Taylor_computable}, are computed using the cG(3) method with 100 finite elements, with the exceptions of \S \ref{sec:example_heat} where cG(3) is used with 40 elements and \S \ref{sec:UQLorenz} where cG(2) with 100 elements is used.  For all methods define $n_{\rm adj}$ to be the number of adjoint solutions required to compute the error in the QoI. This number can be seen as the relative cost of implementing the different methods.

\subsection{Linear problem}
\label{sec:example_linear}

We consider the initial value problem
\begin{equation*}\label{LinEx1}
\dot{y} = \sin(2\pi t)y, \quad \ t \in (0,1], \quad y(0) =1, \\
\end{equation*}
with analytic solution
$$
y(t) = \exp\left( \frac{1}{2 \pi}(1-\cos(2 \pi t)) \right).
$$
Let $R=1.3$ and $S(y(t))=y(t)$. The true QoI is given by
\begin{equation*}
    t_t = Q(y) = \min\limits_{t \in (0,1]}  \text{arg}( y(t)=1.3 )
        = \frac{1}{2 \pi}(\arccos(-2 \pi \ln(1.3)+1)).
\end{equation*}
For this problem, the terms in \eqref{eq:taylor_eta_approx} are
\begin{equation*}
v=1, \; f(y,t)=\sin(2\pi t)y, \; \nabla_y f(y,t)=\sin(2\pi t),
\end{equation*}
hence, for \eqref{eq:taylor_numerator_estimate}, \eqref{eq:taylor_denominator_estimate}, and \eqref{eq:iterative_est} the values needed are
\begin{equation*}
\psi_1=-1, \; \psi_2=\sin(2\pi t_c), \; \psi_3=1.
\end{equation*}
The true solution and QoI are shown in Figure \ref{fig:example_linear}.  This graph includes a horizontal line at $S(y(t))=R$, to indicate the threshold value of interest, as well as a vertical line denoting the true value of the QoI, i.e. the first time the threshold is crossed. Figure   \ref{fig:example_linear} compares the numerical QoI to the true value for both the numerical schemes. True errors, error estimates and effectivity ratios are provided in Tables  \ref{tab:lin_cG} and \ref{tab:lin_CN}. All methods provide excellent effectivity ratios, but the iterative methods require many more applications of Theorem \ref{thm:basic_error} and hence require solving more adjoint problems of the form \eqref{eq:ivp_adjoint}, as shown by the values of $\eta_{adj}$.

\begin{figure}[h!]
    \centering
    \subfloat[Comparing cG(1) solution and computed QoI\eqref{eq:QoI} to the true values for example in \S \ref{sec:example_linear}.]{\includegraphics[width=7cm]{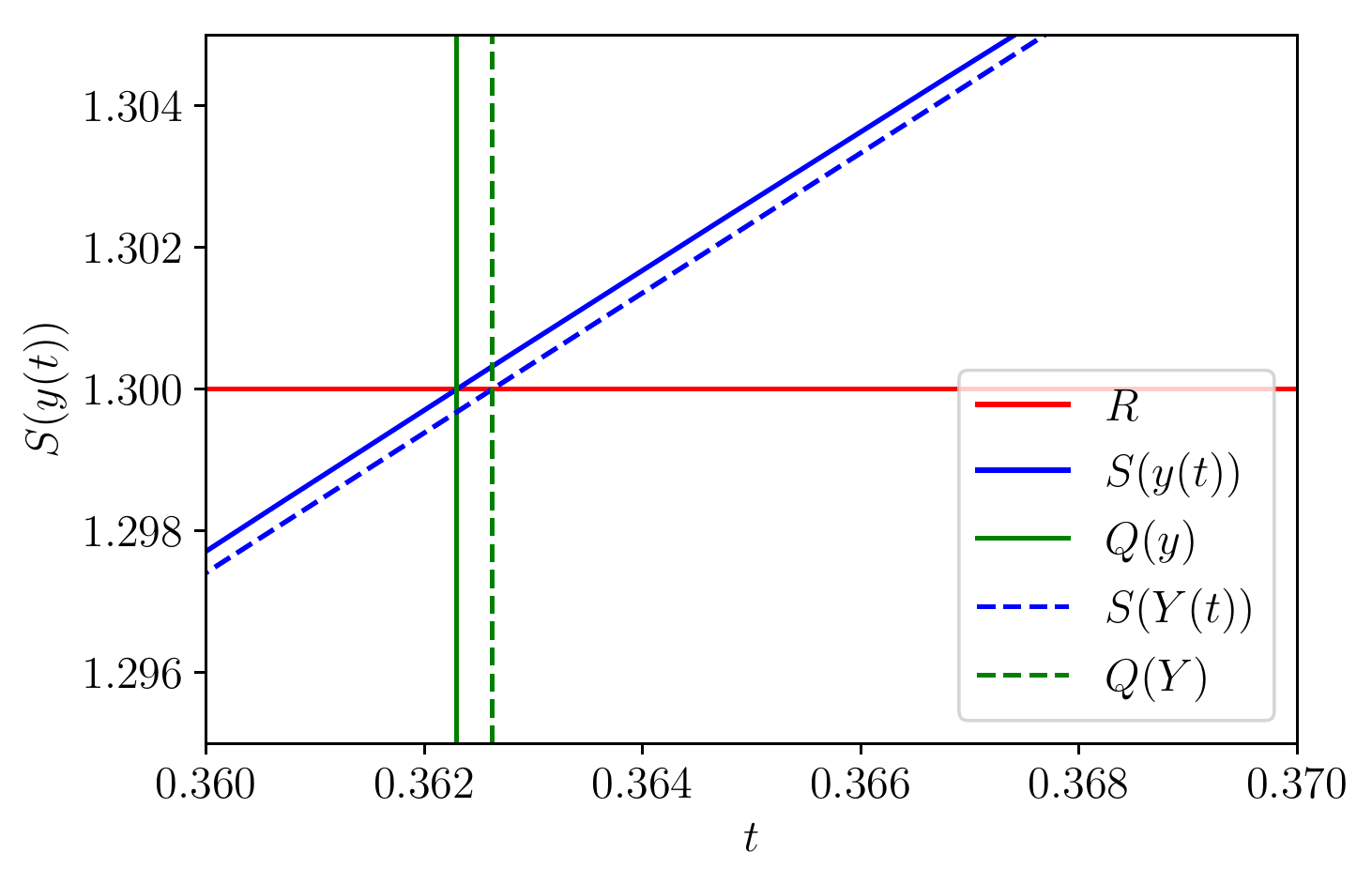}}
    \hfill
    \subfloat[Comparing Crank-Nicolson solution and computed QoI \eqref{eq:QoI} to the true values for example in \S \ref{sec:example_linear}.]{\includegraphics[width=7cm]{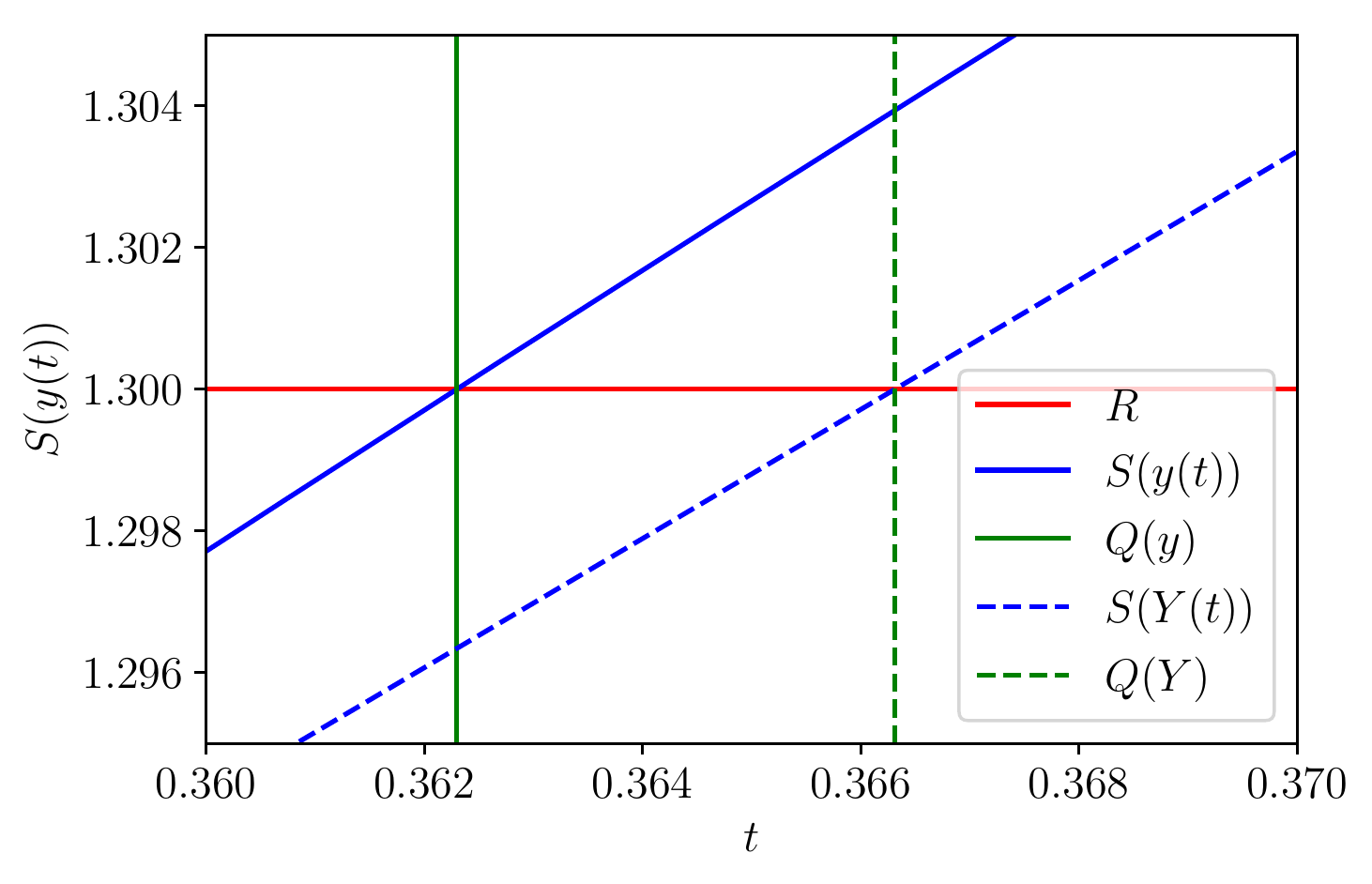}}
    \caption{}
    \label{fig:example_linear}
\end{figure}

\begin{table}[h!]
\centering
\caption{Results of the different methods on the example in \S \ref{sec:example_linear} using cG(1) with 40 elements.}
  \begin{tabular}{||l||c||c|c|c||c|c|c|c||}
  \hline
  Method & $t_c$ & $t_{LL}$ & $t_L$ & $t_R$ & $e_Q$ & $\eta$ & $\rho_{\rm eff}$ & $n_{\rm adj}$ \\
  \hline
  Taylor series  & 0.3626& -- & -- & -- & -3.267e-04 & -3.269e-04 & 1.000 & 2 \\
  Secant         &0.3626 & --  & 0.35  & 0.375 & -3.267e-04 & -3.267e-04 & 1.000 & 6 \\
  Inverse quad.  & 0.3626& 0.325 & 0.35 & 0.375 & -3.267e-04 & -3.267e-04 & 1.000 & 7 \\
  \hline
  \end{tabular}
\label{tab:lin_cG}
\end{table}

\begin{table}[h!]
\centering
  \caption{Results of the different methods on the example in \S \ref{sec:example_linear} using Crank-Nicolson with 21 nodes.}
  \begin{tabular}{||l||c||c|c|c||c|c|c|c||}
  \hline
  Method & $t_c$ & $t_{LL}$ & $t_L$ & $t_R$ & $e_Q$ & $\eta$ & $\rho_{\rm eff}$ & $n_{\rm adj}$ \\
  \hline
    Taylor series  & 0.3663& -- & -- & -- & -4.017e-03 & -4.056e-03 & 1.010 & 2 \\
    Secant         & 0.3663& -- & 0.35  & 0.4 & -4.017e-03 & -4.017e-03 & 1.000 & 7 \\
    Inverse quad. & 0.3663& 0.3 & 0.35 & 0.4 & -4.017e-03 & -4.017e-03 & 1.000 & 7 \\
  \hline
  \end{tabular}
  \label{tab:lin_CN}
\end{table}

\subsection{Nonlinear problem}
\label{sec:example_nonlinear}

Next we consider the nonlinear initial value problem
\begin{equation*}
      \dot{y}(t) = \sin(2\pi y(t)), \; t \in (0,1], \quad  y(0)=\frac{1}{4}.
\end{equation*}
The analytic solution to this problem is
\begin{equation*}
y(t)=\frac{1}{\pi} \arctan (e^{2 \pi t}).
\end{equation*}
Let $R=0.4$ and $S(y(t)) = y(t)$. The true QoI is
\begin{equation*}
    t_t = Q(y) = \min\limits_{t \in [0,1]} \text{arg} (y(t)=0.4)
        = \frac{\ln(\tan(0.4\pi))}{2\pi}.
\end{equation*}
Here, the terms in \eqref{eq:taylor_eta_approx} are
\begin{equation*}
v=1, \; f(y,t)=\sin(2\pi y), \; \nabla_y f(y,t)=2\pi\cos(2\pi y),
\end{equation*}
so the data needed for \eqref{eq:taylor_numerator_estimate}, \eqref{eq:taylor_denominator_estimate}, and \eqref{eq:iterative_est} are
\begin{equation*}
\psi_1= -1, \; \psi_2=  2\pi \cos(2\pi R), \; \psi_3= 1.
\end{equation*}
Figure \ref{fig:true2} shows the true values of the linear functional $S(y(t))$ as well as the event in question and the true QoI. The values in Tables  \ref{tab:nonlin_CG} and \ref{tab:nonlin_CN} indicate that all three methods are fairly accurate. The two iterative methods again require more adjoint equations to be solved.

\begin{figure}[h!]
    \centering
    \subfloat[Chosen value of R, true data $S(y(t)$, and true QoI for example in \S \ref{sec:example_nonlinear}.]{\includegraphics[width=7cm]{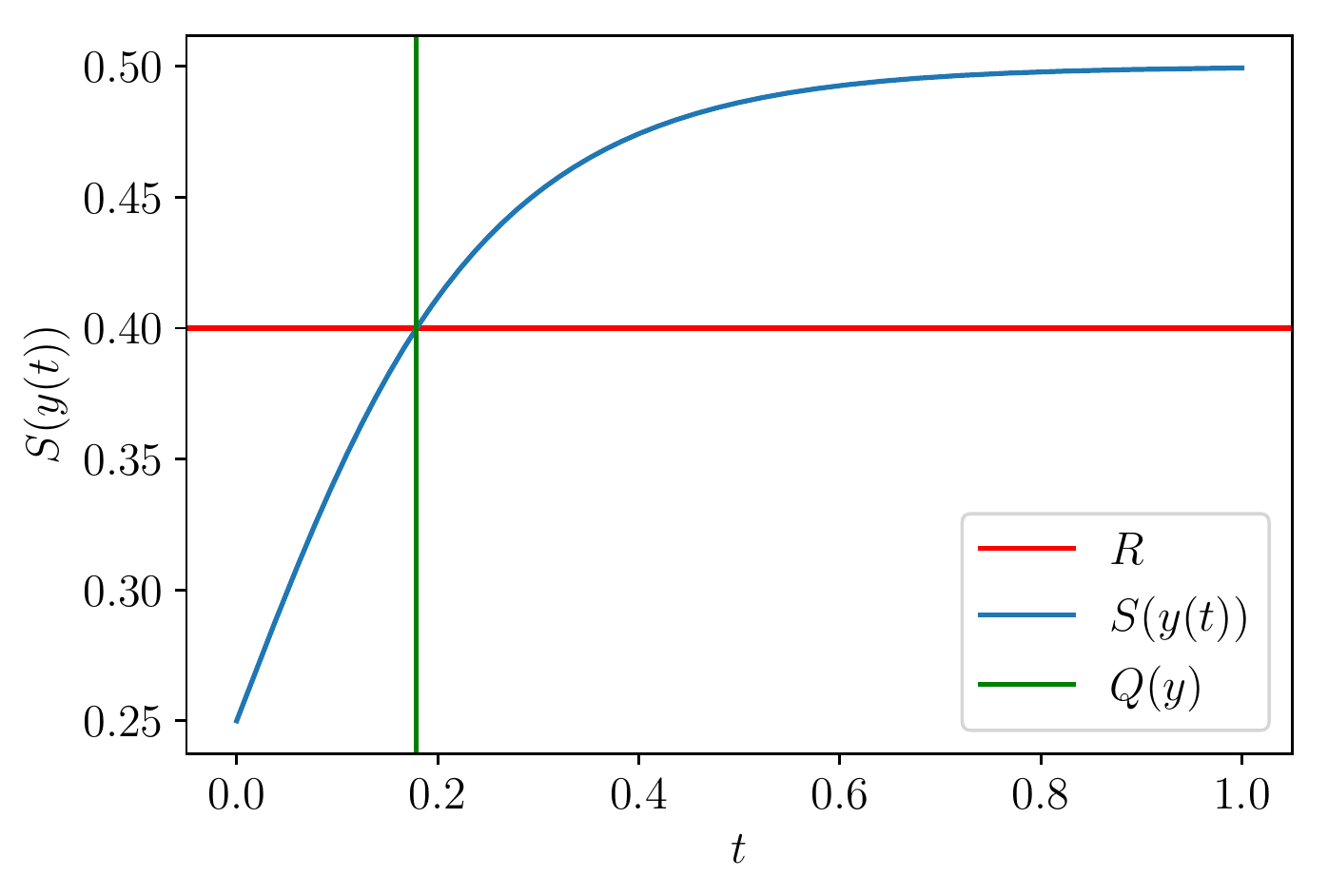} \label{fig:true2}}
    \hfill
    \subfloat[Chosen value of R, true data $S(y(t)$, and true QoI for example in \S \ref{sec:example_linear_system}]{\includegraphics[width=7cm]{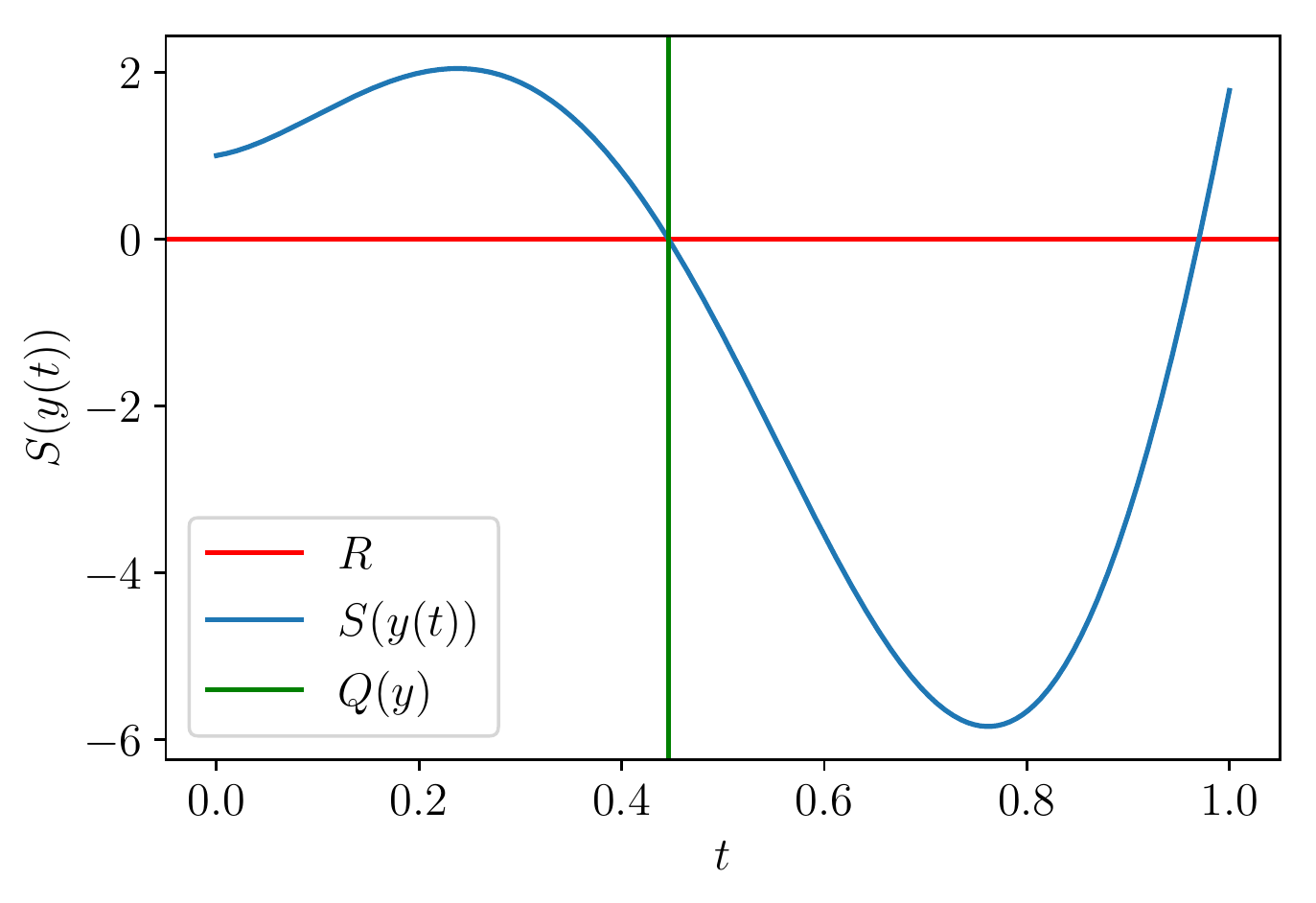} \label{fig:true3}}
    \caption{}

\end{figure}

\begin{table}[h!]
\centering
  \caption{Results for \S \ref{sec:example_nonlinear} using the different methods on cG(1) solution with 40 elements.}
  \begin{tabular}{||l||c||c|c|c||c|c|c|c||}
  \hline
  Method & $t_c$ & $t_{LL}$ & $t_L$ & $t_R$ & $e_Q$ & $\eta$ & $\rho_{\rm eff}$ & $n_{\rm adj}$ \\
  \hline
  Taylor series  & 0.1790& --  & -- & -- & -1.087e-04 & -1.086e-04 & 1.000 & 2 \\
  Secant         & 0.1790& -- & 0.175 & 0.2 & -1.087e-04 & -1.087e-04 & 1.000 & 6 \\
  Inverse quad. & 0.1790& 0.15  & 0.175 & 0.2 & -1.087e-04 & -1.087e-04 & 1.000 & 6 \\
  \hline
  \end{tabular}
  \label{tab:nonlin_CG}
\end{table}

\begin{table}[h!]
\centering
  \caption{Results for \S \ref{sec:example_nonlinear} using the different methods on Crank-Nicolson solution with 21 nodes.}
  \begin{tabular}{||l||c||c|c|c||c|c|c|c||}
  \hline
  Method & $t_c$ & $t_{LL}$ & $t_L$ & $t_R$ & $e_Q$ & $\eta$ & $\rho_{\rm eff}$ & $n_{\rm adj}$ \\
  \hline
  Taylor series &0.1810 & -- & --& --& -2.156e-03 & -2.141e-03 & 1.007 & 2 \\
  Secant        &0.1810 & -- & 0.15 & 0.2 & -2.156e-03 & -2.144e-03 & 1.001 & 7 \\
  Inverse quad. & 0.1810& 0.1  & 0.15 & 0.2 & -2.156e-03 & -2.144e-03 & 1.001 & 7 \\
  \hline
  \end{tabular}
  \label{tab:nonlin_CN}
\end{table}

\subsection{Linear system}
\label{sec:example_linear_system}

We consider the  two dimensional system  $\dot{y} + A(t)y = 0$,
\begin{equation*}
\begin{pmatrix}
  \dot{y_1}(t) \\ \dot{y_2}(t)
\end{pmatrix}
+
\begin{pmatrix}
1 + 9\cos^2(6t)-6\sin(12t) & -12\cos^2(6t)-9/2\sin(12t) \\
12\sin^2(6t)-9/2\sin(12t)  &  1 + 9\sin^2(6t)+6\sin(12t)
\end{pmatrix}
\begin{pmatrix}
y_1(t)  \\ y_2(t)
\end{pmatrix}
=
\begin{pmatrix}
0 \\ 0
\end{pmatrix},
\ t \in (0,1],
\end{equation*}
with initial conditions $y_1(0)=y_2(0)=1$. The analytic solution to this problem is
\begin{equation*}
\begin{pmatrix}
y_1(t) \\ y_2(t)
\end{pmatrix}
=
\begin{pmatrix}
3/5 \exp(2t)(\cos(6t)+2\sin(6t))-1/5 \exp(-13t)(\sin(6t)-2\cos(6t)) \\
3/5 \exp(2t)(2\cos(6t)-\sin(6t))-1/5 \exp(-13t)(\cos(6t)+2\sin(6t))
\end{pmatrix}
.
\end{equation*}
Set $R=0$ and $S(y(t))=y_1(t)$ in order to analyze the first component. The true quantity of interest is
\begin{equation*}
    t_{t} := Q(y) = 0.446255366908554
\end{equation*}
The parameters needed for \eqref{eq:taylor_eta_approx} are
\begin{equation*}
v=(1,0)^\top, \; f(y,t)=-A(t)y, \; \nabla_yf(y,t)=-A(t).
\end{equation*}
For \eqref{eq:taylor_numerator_estimate}, \eqref{eq:taylor_denominator_estimate}, and \eqref{eq:iterative_est} the values needed are
\begin{equation*}
\psi_1= -(1,0)^\top, \;
\psi_2 = ( 1 + 9\cos^2(6t_c)-6\sin(12t_c), \, -12\cos^2(6t_c)-\frac{9}{2}\sin(12t_c) )^\top, \;
\psi_3=  (1,0)^\top.
\end{equation*}
The true solution and QoI are shown in Figure \ref{fig:true3}.
Tables \ref{tab:syst1_cG} and \ref{tab:syst1_CN} show the results for cG(1) and Crank-Nicolson respectively. Again, all methods are accurate using either numerical method. The two iterative methods require many more adjoint problems to be solved than the Taylor series method without  any increase in accuracy.
\begin{table}[h!]
\centering
  \caption{Results of the different methods on example in \S \ref{sec:example_linear_system} using cG(1) with 40 elements.}
  \begin{tabular}{||l||c||c|c|c||c|c|c|c||}
  \hline
  Method & $t_c$ & $t_{LL}$ & $t_L$ & $t_R$ & $e_Q$ & $\eta$ & $\rho_{\rm eff}$ & $n_{\rm adj}$ \\
  \hline
  Taylor series &0.4463 & --  & -- & -- & -1.323e-04 & -1.322e-04 & 0.999 & 2 \\
  Secant method & 0.4463& -- & 0.425 & 0.45 & -1.323e-04 & -1.323e-04 & 1.000 & 6 \\
  Inverse quad. &0.4463 & 0.4   & 0.425 & 0.45 & -1.323e-04 & -1.323e-04 & 1.000 & 8 \\
  \hline
  \end{tabular}
  \label{tab:syst1_cG}
\end{table}

\begin{table}[h!]
\centering
  \caption{Results of the different methods on example in \S \ref{sec:example_linear_system} using Crank-Nicolson with 21 nodes.}
  \begin{tabular}{||l||c||c|c|c||c|c|c|c||}
  \hline
  Method & $t_c$ & $t_{LL}$ & $t_L$ & $t_R$ & $e_Q$ & $\eta$ & $\rho_{\rm eff}$ & $n_{\rm adj}$ \\
  \hline
  Taylor series & 0.4462& -- & -- & -- & 2.675e-05 & 2.675e-05 & 1.000 & 2 \\
  Secant        & 0.4462& --  & 0.4 & 0.45 & 2.675e-05 & 2.675e-05 & 1.000 & 6 \\
  Inverse quad. & 0.4462& 0.35 & 0.4 & 0.45 & 2.675e-05 & 2.675e-05 & 1.000 & 8 \\
  \hline
  \end{tabular}
  \label{tab:syst1_CN}
\end{table}

\subsection{Harmonic oscillator}
\label{sec:example_harmonic}

Consider the harmonic oscillator
\begin{equation*}
\ddot{\omega}=
  -\frac{k}{m}\omega-\frac{c}{m}\dot{\omega}+\frac{F_0}{m}\cos(\gamma t +\theta_d), \ t \in (0,2], \quad \omega(0)=5, \ \dot{\omega}(0)=0.
\end{equation*}
with
\begin{equation*}
k=50, \ m=0.25, \ c=1, \ F_0=50, \ \theta_d=0, \ \gamma=10.
\end{equation*}
Rewriting as a system of first-order ODEs,  $\dot{y}+Ay=h(t)$, gives
\begin{equation*}
\begin{pmatrix}
\dot{y_1}(t) \\  \dot{y_2}(t)
\end{pmatrix}
+
\begin{pmatrix}
0    & -1 \\  200  & 4
\end{pmatrix}
\begin{pmatrix}
y_1(t) \\ y_2(t)
\end{pmatrix}
=
\begin{pmatrix}
0 \\  200\cos(10 t)
\end{pmatrix}.
\end{equation*}
Set $R=0$ and $S(y(t))=y_1(t)$ in order to observe when the oscillator first reaches the origin.
The true solution in \cite{Barger} is used to determine
\begin{equation*}
t_{t} := Q(\omega) = 0.14034864129073557.
\end{equation*}
Here for \eqref{eq:taylor_eta_approx}, the values needed are
\begin{equation*}
v=(1,0)^\top, \; f(y,t)=-Ay+h(t), \;  \nabla_y f(y,t)=-A.
\end{equation*}
To compute \eqref{eq:taylor_numerator_estimate}, \eqref{eq:taylor_denominator_estimate}, and \eqref{eq:iterative_est}, let
\begin{equation*}
\psi_1= -(1,0)^\top, \; \psi_2= (0,1)^\top, \; \psi_3= (1,0)^\top.
\end{equation*}
The true data $S(y(t))$ and QoI are given in Figure \ref{true4} and the results using cG(1) and Crank-Nicolson method are provided in Tables \ref{tab:osc_cG} and \ref{tab:osc_CN} respectively. All methods using either numerical method give effectivity ratios close to one. The two iterative methods require more adjoint problems to be solved than the Taylor series estimate, but they do lead to a slightly more accurate error estimate.
\begin{figure}[h!]
    \centering
    \subfloat[Chosen value of R, true data $S(y(t)$, and true QoI $Q(y)$ for example \ref{sec:example_harmonic}]{\includegraphics[width=7cm]{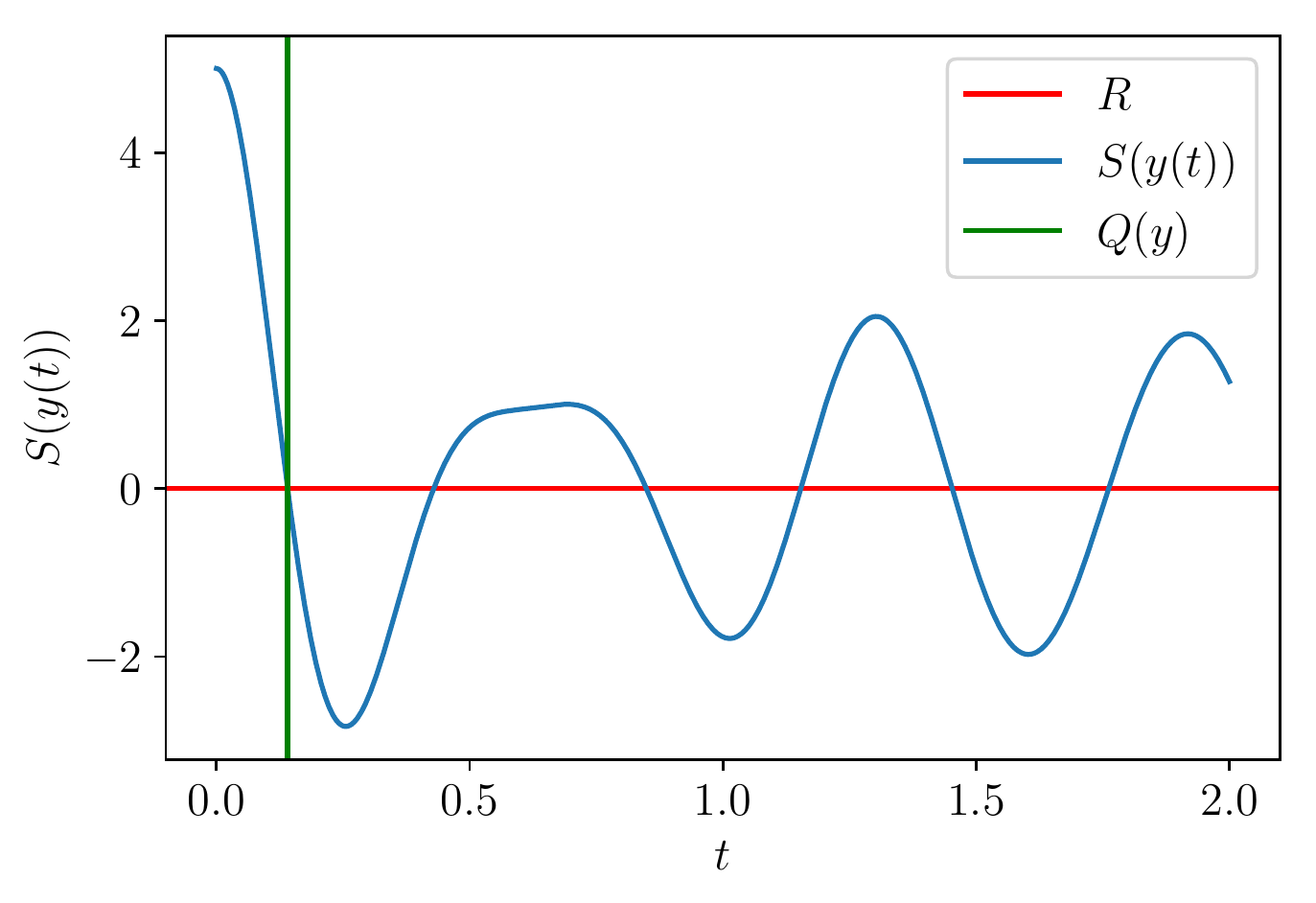} \label{true4} }
    \hfill
    \subfloat[Chosen value of R, true data $S(y(t))$, and true QoI $Q(y)$ for example in \S \ref{sec:example_harmonic_interval}]{\includegraphics[width=7cm]{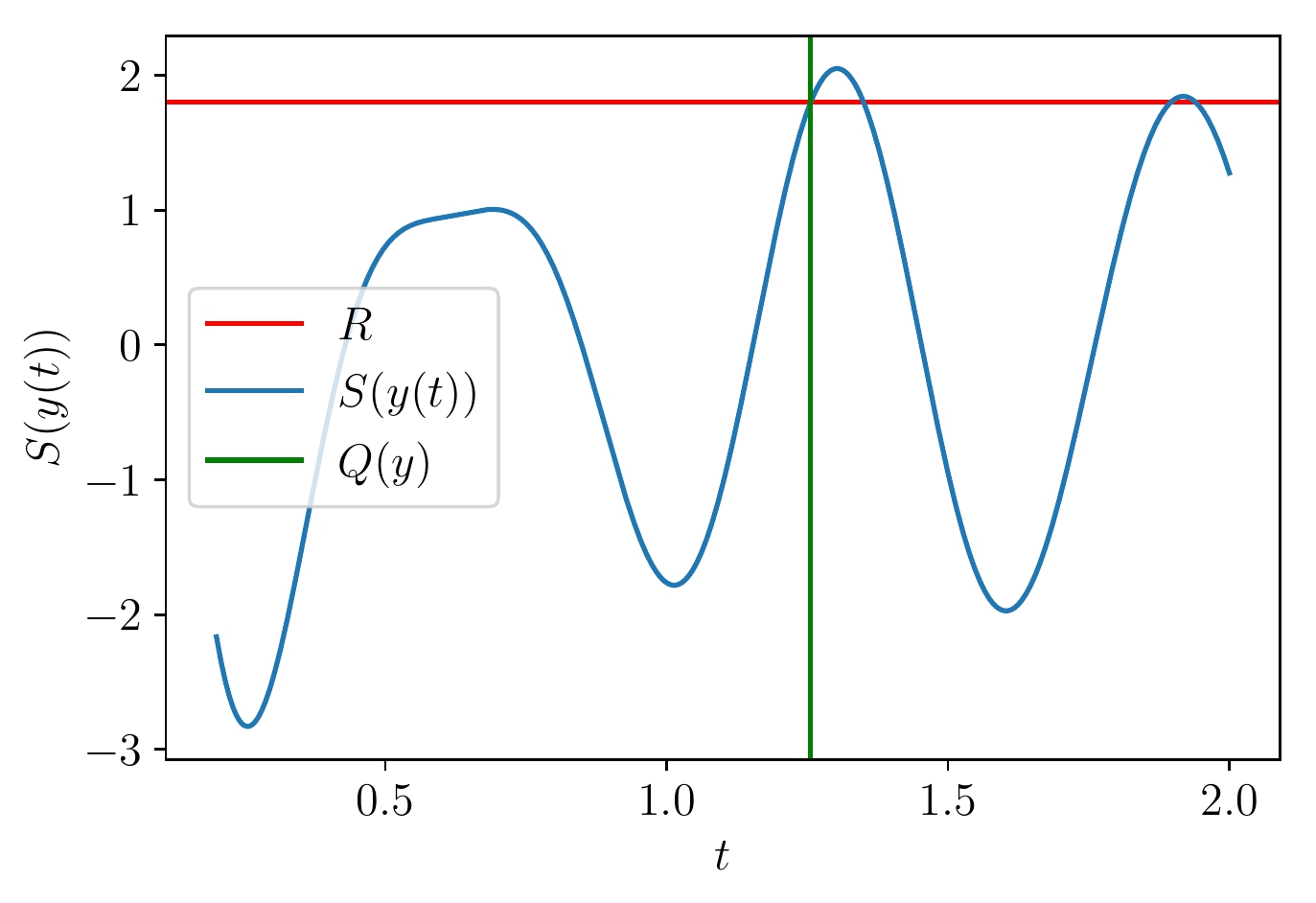} \label{fig:shiftosc}}

    \caption{}

\end{figure}

\begin{table}[h!]
\centering
  \caption{Results of the different methods on the example in \S \ref{sec:example_harmonic} using cG(1) with 40 elements.}
  \begin{tabular}{||l||c||c|c|c||c|c|c|c||}
  \hline
  Method & $t_c$ & $t_{LL}$ & $t_L$ & $t_R$ & $e_Q$ & $\eta$ & $\rho_{\rm eff}$ & $n_{\rm adj}$ \\
  \hline
  Taylor series &0.1447 & -- & -- & -- & -4.440e-03 & -4.449e-03 & 1.011 & 2 \\
  Secant method &0.1447 & --  & 0.1 & 0.15 & -4.440e-03 & -4.440e-03 & 1.000 & 7 \\
  Inverse quad. &0.1447 & 0.05 & 0.1 & 0.15 & -4.440e-03 & -4.440e-03 & 1.000 & 8 \\
  \hline
  \end{tabular}
  \label{tab:osc_cG}
\end{table}

\begin{table}[h!]
\centering
  \caption{Results of the different methods on the example in \S \ref{sec:example_harmonic} using Crank-Nicolson with 21 nodes.}
  \begin{tabular}{||l||c||c|c|c||c|c|c|c||}
  \hline
  Method & $t_c$ & $t_{LL}$ & $t_L$ & $t_R$ & $e_Q$ & $\eta$ & $\rho_{\rm eff}$ & $n_{\rm adj}$ \\
  \hline
  Taylor series &0.1575 &--  & -- & -- & -1.715-02 & -1.816e-02 & 1.059 & 2 \\
  Secant method & 0.1575& -- & 0.1 & 0.2 & -1.715-02 & -1.715e-02 & 0.999 & 8 \\
  Inverse quad. &0.1575 & 0.0 & 0.1 & 0.2 & -1.715-02 & -1.715e-02 & 0.999 & 10 \\
  \hline
  \end{tabular}
  \label{tab:osc_CN}
\end{table}

\subsubsection{Harmonic Oscillator: Effect of the choice of interval}
\label{sec:example_harmonic_interval}

We consider the same equation and function as in \S \ref{sec:example_harmonic}, except over the time interval $t \in (0.2, 2]$ and with $R=1.8$.

Applying the secant method to the true solution results in the true QoI,
\begin{equation*}
    t_t = 1.2558594599461572.
\end{equation*}

Since this problem has the same ODE and functional $S$ as in \S \ref{sec:example_harmonic}, the parameters and steps laid out in that section can be used to obtain the error estimates.

The true functional and QoI are shown in Figure \ref{fig:shiftosc} and the results when using the different methods in Tables \ref{tab:shiftosc_CG} and \ref{tab:shiftosc_CN}. The Taylor series method is slightly less accurate compared to the iterative methods when using the cG(1) method. This is due to the size of the second derivative of the functional near the event, leading to a larger absolute value of the remainder in \eqref{secondapprox}. Since the error estimate \eqref{eq:taylor_eta_approx} neglects this remainder, if its absolute value is too large the estimate will not be accurate. Examples in \S \ref{sec:example_harmonic_R} take a further look into this effect.

Both the Taylor series and iterative methods are poor for the Crank-Nicolson method. This is due to the low accuracy of the numerical solution as illustrated in Figure \ref{fig:shiftosc_CN}. The potential inaccuracy of the Taylor series estimate under these circumstances is discussed in Remark~\ref{rem:err_est_accuracy}. The root-finding methods are converging to the \emph{second} time the event occurs (which is 1.3237), rather than the first. Because of the small difference in time between the locations of the two roots (see Figure \ref{fig:iter_issues}), the proximity of the second root to the numerical QoI, and the size of the numerical time step, both roots are contained within the initial interval over which the iterative methods are applied. It is therefore possible for the iterative methods to converge to the larger of the two roots.

\begin{figure}[h!]
    \centering
    \subfloat[Figure detailing issue with iterative methods in \S \ref{sec:example_harmonic_interval} and \S \ref{sec:example_harmonic_R} when the numerical solution is not accurate near the event. The iterative methods result in $t^* = \eta_{it}$, which is the second occurrence of the event rather than the first. This figure specifically details the case when $R=2$. ]{\includegraphics[width=7cm]{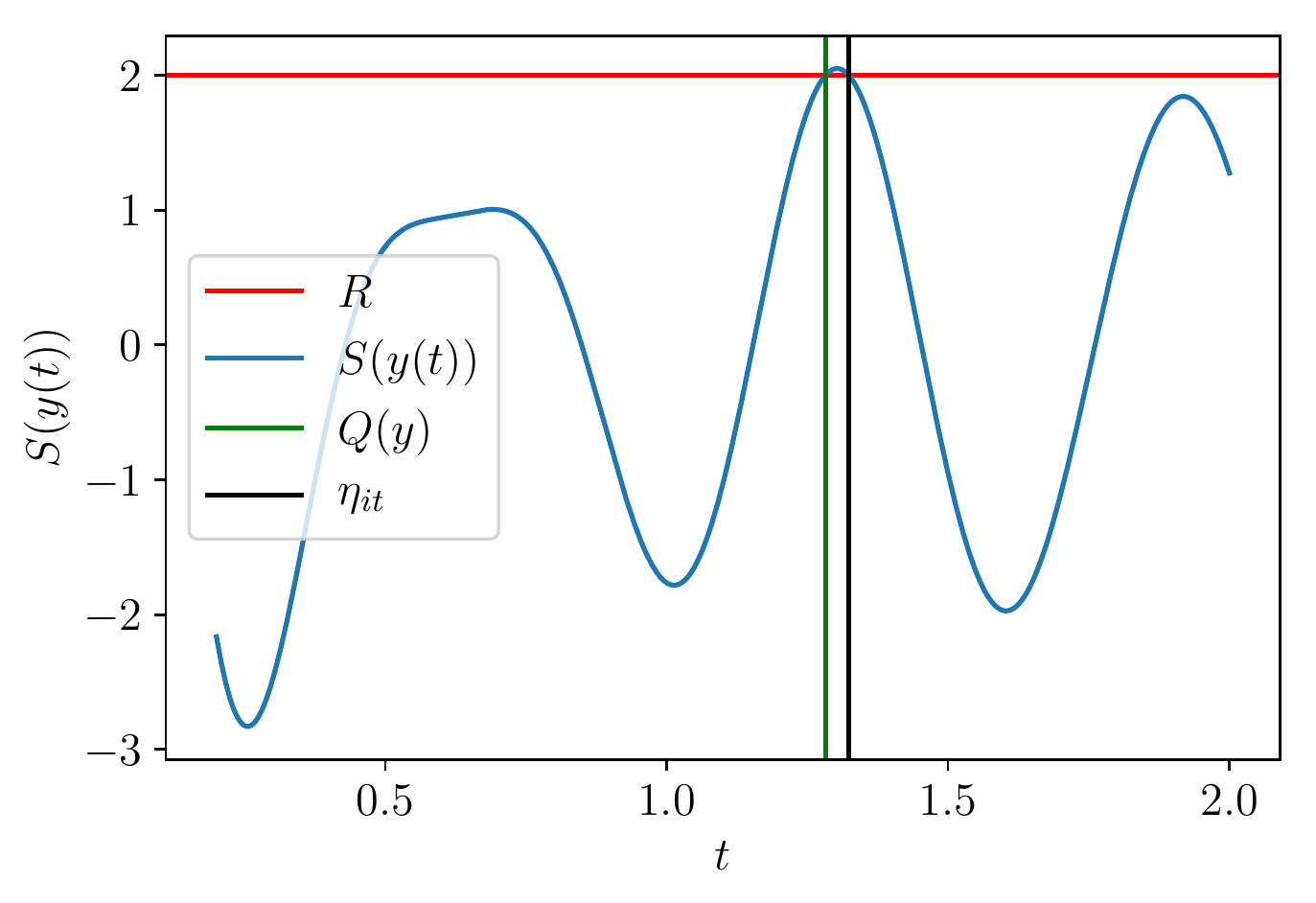} \label{fig:iter_issues}}
    \hfill
    \subfloat[Numerical values for example in \S \ref{sec:example_harmonic_interval} when using Crank-Nicolson method with 21 nodes.]{\includegraphics[width=7cm]{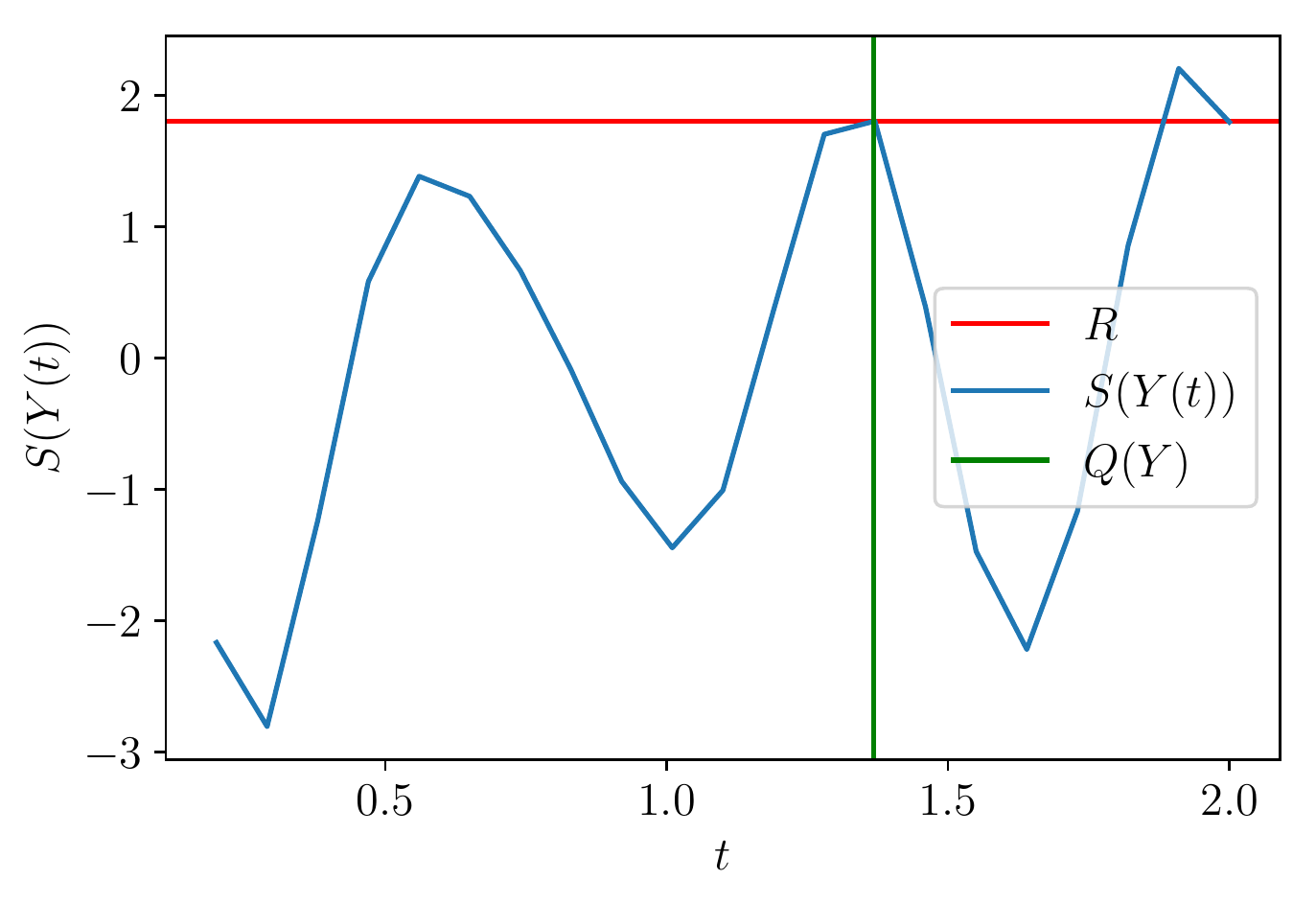} \label{fig:shiftosc_CN}}
    \caption{}
\end{figure}

\begin{table}[h!]
\centering
  \caption{Results of the different methods on example in \S \ref{sec:example_harmonic_interval} using cG(1) with 40 elements.}
  \begin{tabular}{||l||c||c|c|c||c|c|c|c||}
  \hline
  Method & $t_c$ & $t_{LL}$ & $t_L$ & $t_R$ & $e_Q$ & $\eta$ & $\rho_{\rm eff}$ & $n_{\rm adj}$ \\
  \hline
  Taylor series & 1.2637&--   &--  &--  & -7.887e-03 & -8.623e-03 & 1.093 & 2 \\
  Secant        &1.2637 &  --& 1.235 & 1.37 & -7.887e-03 & -7.887e-03 & 0.999 & 8 \\
  Inverse quad. & 1.2637& 1.19  & 1.235 & 1.37 & -7.887e-03 & -7.887e-03 & 0.999 & 9 \\
  \hline
  \end{tabular}
  \label{tab:shiftosc_CG}
\end{table}

\begin{table}[h!]
\centering
  \caption{Results of the different methods on example in \S \ref{sec:example_harmonic_interval} using Crank-Nicolson with 21 nodes.}
  \begin{tabular}{||l||c||c|c|c||c|c|c|c||}
  \hline
  Method & $t_c$ & $t_{LL}$ & $t_L$ & $t_R$ & $e_Q$ & $\eta$ & $\rho_{\rm eff}$ & $n_{\rm adj}$ \\
  \hline
  Taylor series & 1.3674& -- & -- & -- & -1.116e-01 & -1.542e-02 & 0.138 & 2 \\
  Secant method & 1.3674& -- & 1.28 & 1.37 & -1.116e-01 & -1.746e-02 & 0.156 & 8 \\
  Inverse quad. & 1.3674& 1.19 & 1.28 & 1.37 & -1.116e-01 & -1.746e-02 & 0.156 & 10 \\
  \hline
  \end{tabular}
  \label{tab:shiftosc_CN}
\end{table}

\subsubsection{Harmonic oscillator: Effect of the choice of $R$}
\label{sec:example_harmonic_R}

Again consider the harmonic oscillator of \S \ref{sec:example_harmonic_interval}, and estimate the error of the QoI \eqref{eq:QoI} with several different values of $R$, increasing $R$ until it is very close to the maximum of the true data. The maximum value of the true data is approximately 2.05015. Results are provided in Tables \ref{tab:lim_ex1_1}, \ref{tab:lim_ex1_2} and \ref{tab:lim_ex1_3} for increasingly fine finite element meshes.  The tables contain the effectivity ratios, $\rho_{\rm eff}$, for each method and each value of R.

Notice that the iterative methods appear to be more sensitive to the accuracy of the numerical solution than the Taylor series method. In extreme cases, the iterative methods fail to converge. This occurs when a root-finding iteration falls outside of the domain of the IVP \eqref{eq:ivp}, i.e., if $x_n$ the approximation to the root at the $n$th iteration, $x_n < 0$ or $x_n > T$. As the number of finite elements used to solve the ODE increases, the two iterative methods eventually recover their accuracy even when the threshold value is very close to an extremum. For the cases where the iterative methods are inaccurate, note that the root-finding schemes do \textit{not} converge to the true QoI. Instead, the convergence is to the \textit{second} occurrence of the event rather than the first (see Figure \ref{fig:iter_issues}).

The estimate derived from Taylor's theorem is generally more accurate for the less accurate numerical solutions, However, even when using an accurate numerical solution, \textcolor{black}{the Taylor} series approach becomes inaccurate when the curvature of $S$ as a function of $t$ is large near the threshold value. The remainder $\mathcal{R}_1(t_t,t_c)$, given by \eqref{eq:Taylor_remain1}, is one half of the second derivative of $S$ with respect to $t$ at some point between  $t_t$ and $t_c$. As the threshold value $R$ moves closer to the local maximum, this derivative grows and the assumption that  $\mathcal{R}_1(t_t,t_c)$ is small is no longer valid, resulting in an inaccurate estimate.
The iterative methods do not depend on the values of the second derivative of the solution and those methods are able to produce accurate error estimates provided the numerical solution is sufficiently accurate near the event.

\begin{figure}[h!]
    \centering
    \includegraphics[width=7.5cm]{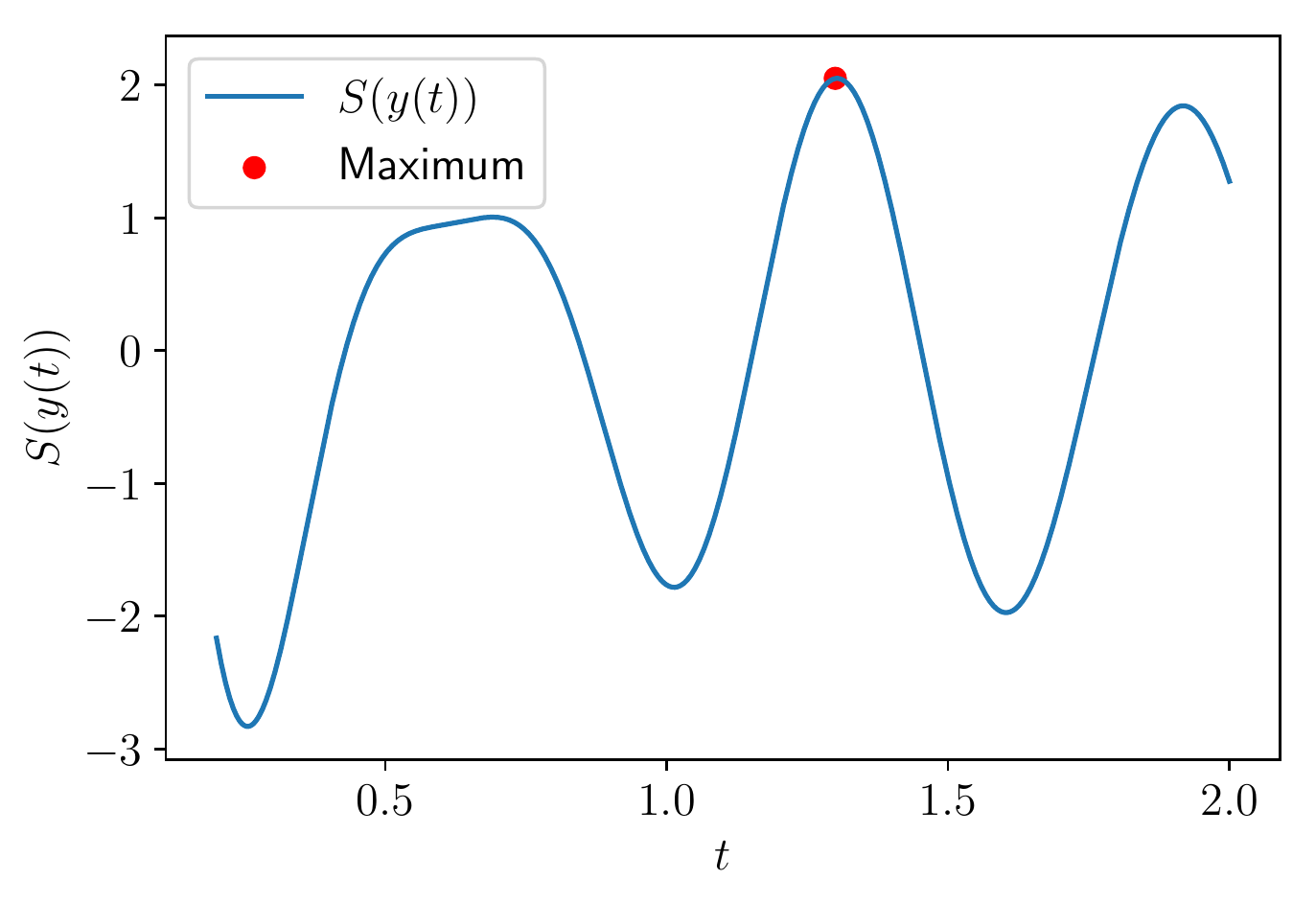}
    \caption{True data for example in \S \ref{sec:example_harmonic_interval}, showing max value of $\approx$ 2.05015.}
    \label{fig:osc_max}
\end{figure}

\begin{table}[h!]
\centering
  \caption{Effectivity ratio for the different methods for varying values of $R$ on example in \S \ref{sec:example_harmonic} using cG(1) with 40 elements.}
  \begin{tabular}{||c|c|c|c|c|c|c|c||}
  \hline
  Method & R=1.95 & R=2.0 & R=2.01 & R=2.02 & R=2.03 & R=2.04 & R=2.05 \\
  \hline
  Taylor series & 1.061 &   1.095 &  1.251 &  1.603 &  3.470 & -1.137 & 0.427 \\
  Secant method & 0.999 & -11.305 & -4.952 & -2.650 & -1.405 &  1.000 & fail  \\
  Inverse quad. & 0.999 & -11.305 & -4.952 & -2.650 & -1.405 &  fail  & fail  \\
  \hline
  \end{tabular}
  \label{tab:lim_ex1_1}
\end{table}

\begin{table}[h!]
\centering
  \caption{Effectivity ratio for  the different methods for varying values of $R$ on example in \S \ref{sec:example_harmonic} using cG(1) with 60 elements.}
  \begin{tabular}{||c|c|c|c|c|c|c|c||}
  \hline
  Method & R=1.95 & R=2.0 & R=2.01 & R=2.02 & R=2.03 & R=2.04 & R=2.05 \\
  \hline
  Taylor series &  1.033 & 0.999 & 1.043 & 1.100 &  1.179 &  1.283 & 0.758 \\
  Secant        &  1.000 & 0.999 & 0.999 & 0.999 & -6.545 & -4.520 & 3.133 \\
  Inverse quad. &  1.000 & 0.999 & 0.999 & 0.999 & -6.545 & -4.520 & 3.133 \\
  \hline
  \end{tabular}
  \label{tab:lim_ex1_2}
\end{table}

\begin{table}[h!]
\centering
  \caption{Effectivity ratio for  the different methods for varying values of $R$ on example in \S \ref{sec:example_harmonic} using cG(1) with 100 elements.}
  \begin{tabular}{||c|c|c|c|c|c|c|c||}
  \hline
	Method & R=1.95 & R=2.0 & R=2.01 & R=2.02 & R=2.03 & R=2.04 & R=2.05 \\
    \hline
    Taylor series &  1.017 & 1.001 & 1.019 & 1.100 & 1.039 & 0.998 & 0.588 \\
    Secant        &  0.999 & 0.999 & 0.999 & 1.000 & 0.999 & 0.999 & 0.999 \\
    Inverse quad. &  0.999 & 0.999 & 0.999 & 1.000 & 0.999 & 0.999 & 0.999 \\
  \hline
  \end{tabular}
  \label{tab:lim_ex1_3}
\end{table}

\subsection{One dimensional heat equation}
\label{sec:example_heat}

We consider the one dimensional heat equation with boundary and initial conditions
\begin{equation*}\label{heat}
\begin{aligned}
&u_t(x,t) = u_{xx}(x,t)+\textcolor{black}{3e^t\sin(\pi x)}, \quad (x,t) \in (0,1) \times (0,1], \\
&u(x,0) = 0, \quad x \in (0,1), \\
&u(0,t) = 0, \; u(1,t) = 0, \quad t \in (0,1].
\end{aligned}
\end{equation*}
This section analyzes the system of ordinary differential equations that arises from a spatial discretization of (\ref{heat}) using a central-difference method. In particular using a uniform partition of the spatial interval $[0,1]$ with 22 nodes:
\begin{equation*}
\{0=x_0<x_1<\dots<x_{21}=1 \}.
\end{equation*}
Since boundary values are specified, this semi-discretization leads to a system of 20 first-order ODEs of the form $\dot{y}(t)=Ay(t)+k(t)$, where $h=\frac{1}{21}$ and

\begin{equation*}
A = \frac{1}{h^2}
\begin{pmatrix}
-2 & 1 & 0 & \cdots & \cdots & 0 \\
1 & -2 & 1 & 0 & \cdots &0 \\
0 & 1 & -2 & 1 & \cdots &0 \\
\vdots & &\ddots &\ddots &\ddots &\vdots & \\
0 & \cdots & 0 & 1 & -2 & 1 \\
0& \cdots & \cdots & 0 & 1 &-2
\end{pmatrix}, \ \ \
k(t)=\textcolor{black}{
\begin{pmatrix}
3e^t\sin(\pi x_1) \\
3e^t\sin(\pi x_2) \\
3e^t\sin(\pi x_3) \\
\vdots \\
3e^t\sin(\pi x_{19}) \\
3e^t\sin(\pi x_{20})
\end{pmatrix}}
\end{equation*}
Since this problem will only analyze the semi-discrete system and not the full PDE, a reference solution is obtained using an accurate time-integrator (SciPy's solve\_ivp) using an absolute tolerance of $10^{-15}$. Let $R=0.33$ and $S(y(t))=\frac{1}{20}\sum_{i=1}^{20}y_i(t)$ in order to analyze the discrete average of the solution over the spatial domain at a time $t$. This library function also has the capability of tracking when specified events occur, which is used to obtain a reference for the true QoI,
\begin{equation*}
t_{t} = 0.5834435609935992.
\end{equation*}
For this problem, the parameters in \eqref{eq:taylor_eta_approx} are
\begin{equation*}
v=\frac{1}{20}(1,1,\dots,1)^\top, \; f(y,t)=Ay+k(t), \; \nabla_yf(y,t)=A.
\end{equation*}
For \eqref{eq:taylor_numerator_estimate}, \eqref{eq:taylor_denominator_estimate}, and \eqref{eq:iterative_est}, set
\begin{equation*}
\psi_1= -\frac{1}{20}(1,1,\dots,1)^\top, \;
\psi_2= \frac{1}{20h^2}(-1,0,\dots,0,-1)^\top, \;
\psi_3=  \frac{1}{20}(1,1,\dots,1)^\top.
\end{equation*}
The  true solution and QoI are shown in Figure \ref{fig:true5} and the results when using cG(1) or Crank-Nicolson methods are shown in Tables \ref{tab:heat_CG} and \ref{tab:heat_CN} respectively. All methods are accurate using either numerical method. The two iterative methods require more adjoint problems to be solved than the Taylor series estimate without any noticeable increase in accuracy.
\begin{figure}[h!]
    \centering
    \subfloat[Chosen value of R, true data $S(y(t)$, and true QoI for example in \S \ref{sec:example_heat}.]{\includegraphics[width=7cm]{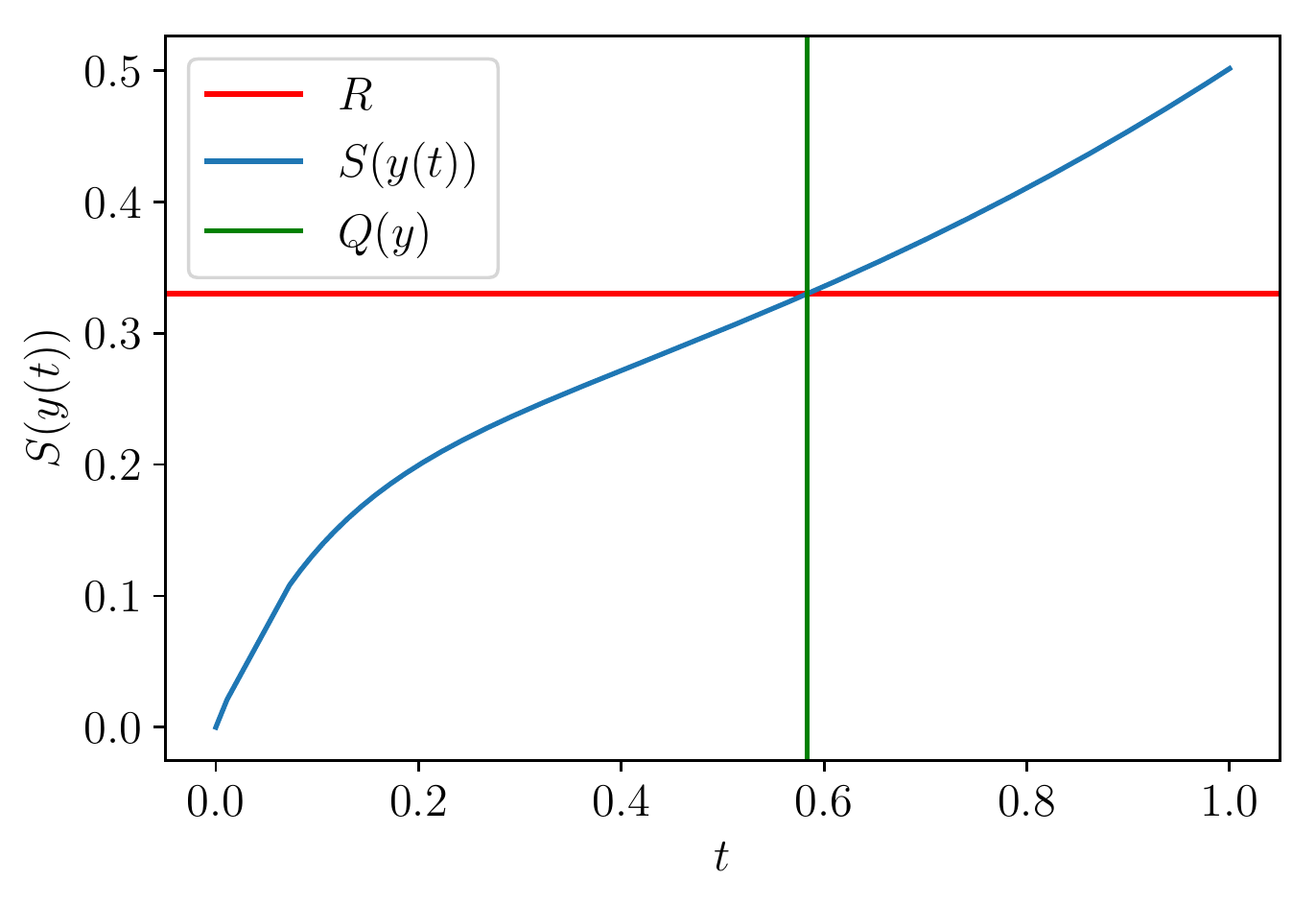} \label{fig:true5}}
    \hfill
    \subfloat[Chosen value of R, true data $S(y(t)$, and true QoI for example in \S \ref{sec:example_two_body}.]{\includegraphics[width=7cm]{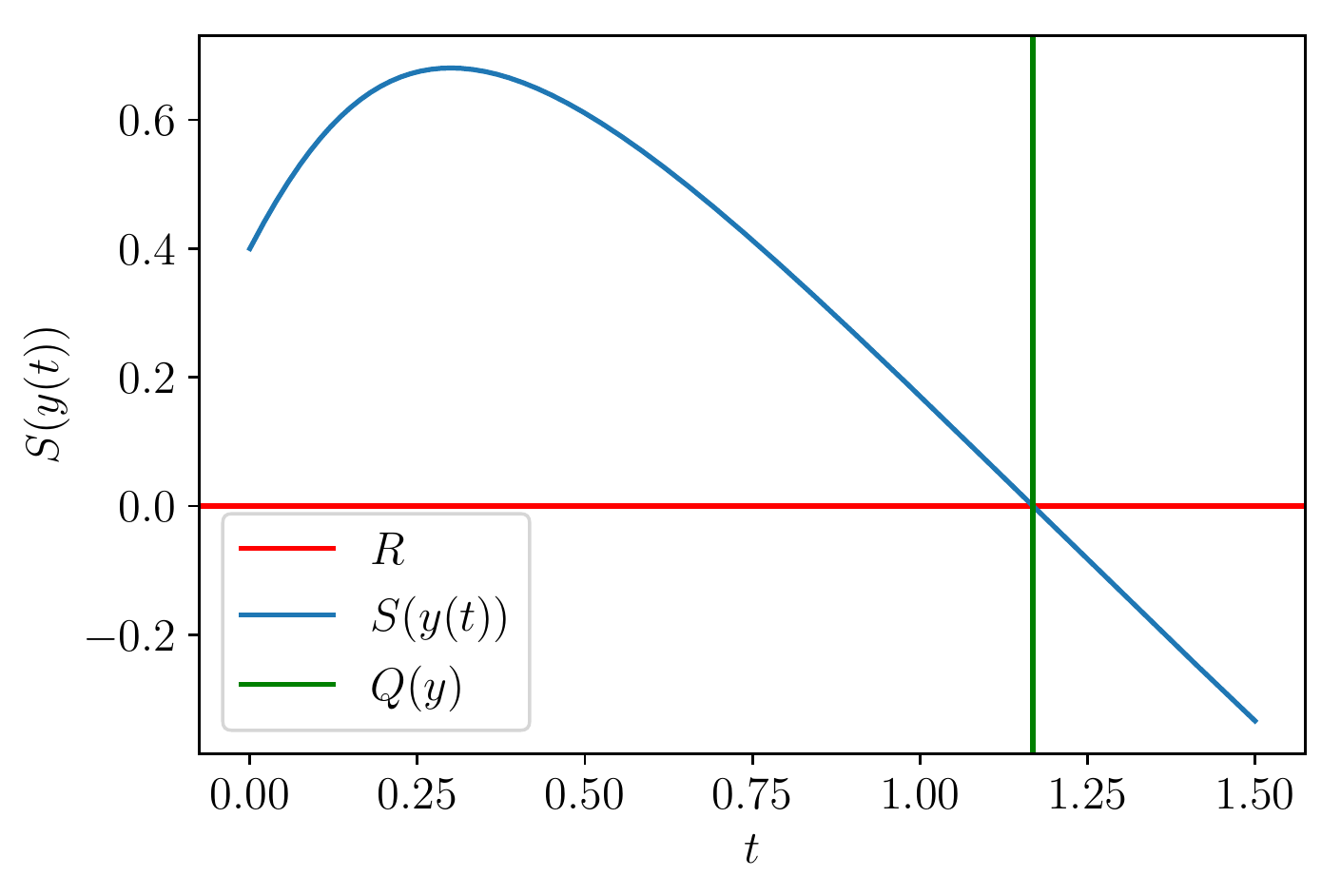} \label{fig:true6}}

    \caption{}

\end{figure}

\begin{table}[h!]
\centering
  \caption{Results of the different methods on the example in \S \ref{sec:example_heat} using cG(1) with 40 elements.}
  \begin{tabular}{||l||c||c|c|c||c|c|c|c||}
  \hline
  Method & $t_c$ & $t_{LL}$ & $t_L$ & $t_R$ & $e_Q$ & $\eta$ & $\rho_{\rm eff}$ & $n_{\rm adj}$ \\
  \hline
  Taylor series & 0.5834 &-- &-- &-- & 6.157e-05 & 6.151e-05 & 0.999 & 2 \\
  Secant        & 0.5834 & -- & 0.575 & 0.6 & 6.157e-05 & 6.150e-05 & 0.999 & 6 \\
  Inverse quad. & 0.5834 & 0.55 & 0.575 & 0.6 & 6.157e-05 & 6.150e-05 & 0.999 & 7 \\
  \hline
  \end{tabular}
  \label{tab:heat_CG}
\end{table}

\begin{table}[h!]
\centering
  \caption{Results of the different methods on the example in \S \ref{sec:example_heat} using Crank-Nicolson with 21 nodes.}
  \begin{tabular}{||l||c||c|c|c||c|c|c|c||}
  \hline
  Method & $t_c$ & $t_{LL}$ & $t_L$ & $t_R$ & $e_Q$ & $\eta$ & $\rho_{\rm eff}$ & $n_{\rm adj}$ \\
  \hline
  Taylor series & 0.5830 & --& --&-- & 4.457e-04 & 4.457e-04 & 1.000 & 2 \\
  Secant        & 0.5830 & -- & 0.55 & 0.6 & 4.457e-04 & 4.456e-04 & 0.999 & 6 \\
  Inverse quad. & 0.5830 & 0.5  & 0.55 & 0.6 & 4.457e-04 & 4.456e-04 & 0.999 & 7 \\
  \hline
  \end{tabular}
  \label{tab:heat_CN}
\end{table}

\subsection{Two body problem}
\label{sec:example_two_body}

We consider the two body problem
\begin{equation*}\label{twobody}
\left.
\begin{gathered}\begin{aligned}
\dot{y}_1 &= y_3, \\
\dot{y}_2 &= y_4, \\
\dot{y}_3 &= \frac{-y_1}{(y_1^2+y_2^2)^{3/2}}, \\
\dot{y}_4 &= \frac{-y_2}{(y_1^2+y_2^2)^{3/2}},
\end{aligned}\end{gathered}
\; \right\}
\quad t \in (0,1.5], \quad y(0)=(0.4, 0, 0, 2.0)^\top,
\end{equation*}
which models a small body orbiting a much larger body in two dimensions. Here $y_1, y_2$ are the spatial coordinates of the orbiting body relative to the larger body, and $y_3, y_4$ are the respective velocities. The initial conditions are chosen so that the analytic solution is \cite{CET2015}
\begin{equation*}
    y = \left( \cos(\tau)-0.6 , 0.8\sin(\tau), \frac{-\sin(\tau)}{1-0.6\cos(\tau)},
               \frac{0.8\cos(\tau)}{1-0.6\cos(\tau)} \right)^\top,
\end{equation*}
where $\tau$ solves $\tau - 0.6\sin(\tau) = t$. Let $R=0$ and $S(y(t))=y_1(t)+y_2(t)$. The true QoI can be found exactly:
\begin{equation*}
    t_t=Q(y)=\cos^{-1}( (15-16\sqrt{2})/41 )
             -0.6\sin\big( \cos^{-1}( (15-16\sqrt{2})/41 ) \big).
\end{equation*}
The values needed to compute \eqref{eq:taylor_eta_approx} are
\begin{equation*}
v= (1,1,0,0)^\top, \;
f(y,t)= \left(y_3, y_4, \frac{-y_1}{(y_1^2+y_2^2)^{3/2}}, \frac{-y_2}{(y_1^2+y_2^2)^{3/2}} \right)^\top,
\end{equation*}
and
\begin{equation*}
\nabla_y f(y,t) =
\begin{pmatrix}
0 & 0 & 1 & 0 \\
0 & 0 & 0 & 1 \\
\frac{2y_1^2-y_2}{(y_1^2+y_2^2)^{5/2}} & \frac{3y_1y_2}{(y_1^2+y_2^2)^{5/2}}    & 0 & 0 \\
\frac{3y_1y_2}{(y_1^2+y_2^2)^{5/2}}    & \frac{2y_1^2-y_2}{(y_1^2+y_2^2)^{5/2}} & 0 & 0 \\
\end{pmatrix}.
\end{equation*}
For  \eqref{eq:taylor_numerator_estimate}, \eqref{eq:taylor_denominator_estimate}, and \eqref{eq:iterative_est}, the data needed are
\begin{equation*}
\psi_1= (-1,-1,0,0)^\top, \;
\psi_2= (0,0,1,1)^\top \;
\psi_3= (1,1,0,0)^\top.
\end{equation*}
The true data $S(y(t))$ and QoI are shown in Figure \ref{fig:true6} and the results using the cG(1) and Crank-Nicolson method appear in Tables \ref{tab:twobod_CG} and \ref{tab:twobod_CN} respectively. All methods have larger error than in other examples so far, probably as a result of the non-linear nature of \eqref{twobody}. However the error estimates are still accurate using either numerical method; each with an effectivity ratio close to one.
\begin{table}[h!]
\centering
  \caption{Results of the different methods on the example in \S \ref{sec:example_two_body} using cG(1) with 40 elements.}
  \begin{tabular}{||l||c||c|c|c||c|c|c|c||}
  \hline
  Method & $t_c$ & $t_{LL}$ & $t_L$ & $t_R$ & $e_Q$ & $\eta$ & $\rho_{\rm eff}$ & $n_{\rm adj}$ \\
  \hline
  Taylor series & 1.1601 &-- & --&-- & 8.262e-03 & 8.287e-03 & 1.003 & 2 \\
  Secant method & 1.1601 & --& 1.125 & 1.1625 & 8.262e-03 & 8.287e-03 & 1.003 & 5 \\
  Inverse quad. & 1.1601 & 1.0875 & 1.125 & 1.1625 & 8.262e-03 & 8.287e-03 & 1.003 & 6 \\
  \hline
  \end{tabular}
  \label{tab:twobod_CG}
\end{table}

\begin{table}[h!]
\centering
  \caption{Results of the different methods on the example in \S \ref{sec:example_two_body} using Crank-Nicolson with 21 nodes.}
  \begin{tabular}{||l||c||c|c|c||c|c|c|c||}
  \hline
  Method & $t_c$ & $t_{LL}$ & $t_L$ & $t_R$ & $e_Q$ & $\eta$ & $\rho_{\rm eff}$ & $n_{\rm adj}$ \\
  \hline
  Taylor series & 1.2091 &-- & --& --& -4.068e-02 & -4.078e-02 & 1.002 & 2 \\
  Secant        & 1.2091 & --& 1.2 & 1.275 & -4.068e-02 & -4.077e-02 & 1.002 & 5 \\
  Inverse quad. & 1.2091 & 1.125 & 1.2 & 1.275 & -4.068e-02 & -4.077e-02 & 1.002 & 6 \\
  \hline
  \end{tabular}
  \label{tab:twobod_CN}
\end{table}
\newpage

\subsection{Logistic Equation}
\label{sec:example_log}
Consider the Logistic equation

\begin{equation}
    \dot{\textcolor{black}{y}} = k\textcolor{black}{y} \left( 1-\frac{\textcolor{black}{y}}{K} \right), \ t \in (0,20], \ \ \textcolor{black}{y}(0)=\frac{1}{2},
\end{equation}
where $k=0.25$ and $K=1$. The analytic solution is,
\begin{equation}
    \textcolor{black}{y}(t)= \frac{K \, \textcolor{black}{y}(0)}{\textcolor{black}{y}(0)+(K-\textcolor{black}{y}(0))e^{-kt}} = \frac{1}{1+3e^{-0.25t}}.
\end{equation}
Let $S(\textcolor{black}{y}(t)) = \textcolor{black}{y}(t)$ and consider several threshold values, $R \in \{0.55,0.8,0.9,\textcolor{black}{0.94},0.98,0.99,0.995 \}$. The values needed for \eqref{eq:taylor_eta_approx} are
\begin{equation*}
v=1, \; f(y,t)=ky(1-\frac{y}{k}), \; \nabla_y f(y,t)=k-\frac{2k}{K}y,
\end{equation*}
so the data needed for \eqref{eq:taylor_numerator_estimate}, \eqref{eq:taylor_denominator_estimate}, and \eqref{eq:iterative_est} are
\begin{equation*}
\psi_1= -1, \; \psi_2=k-\frac{2k}{K}R, \; \psi_3= 1.
\end{equation*}
The numerical solution is computed using the cG(1) method with five elements.
Figure \ref{fig:log_true} shows the true functional and QoI for a chosen threshold value. Table \ref{tab:logistic} shows the true error in the QoI and the effectivity ratio for each method as the threshold value increases. As the error in the QoI increases, the Taylor series method loses accuracy, presumably since the remainder terms are no longer negligible, despite the fact the second derivatives with respect to $t$ are small. However, the iterative methods are accurate even when the true error is large.

\begin{table}[h!]
\centering
  \caption{Error in QoI and effectivity ratio of the different methods for varying values of $R$ on example in \S \ref{sec:example_log} using cG(1) with 5 elements.}
  \begin{tabular}{||c|c|c|c|c|c|c|c||}
  \hline
    & R=0.55 & R=0.8 & R=0.9 & R=0.94 & R=0.98 & R=0.99 & R=0.995 \\
  \hline
  $e_Q$         & -0.090 & -0.117 & -0.166 & 0.194 & 0.829 & 0.610 & 1.513 \\
  Taylor series &  1.001 & 1.021 & 1.041 & 0.957 & 0.902 & 0.919 & 0.830 \\
  Secant        &  0.999 & 0.987 & 0.977 & 1.023 & 1.007 & 1.011 & 1.005 \\
  Inverse quad. &  0.999 & 0.987 & 0.977 & 1.023 & 1.007 & 1.011 & 1.005 \\
  \hline
  \end{tabular}
  \label{tab:logistic}
\end{table}

\begin{figure}[h!]
    \centering
    \includegraphics[width=7cm]{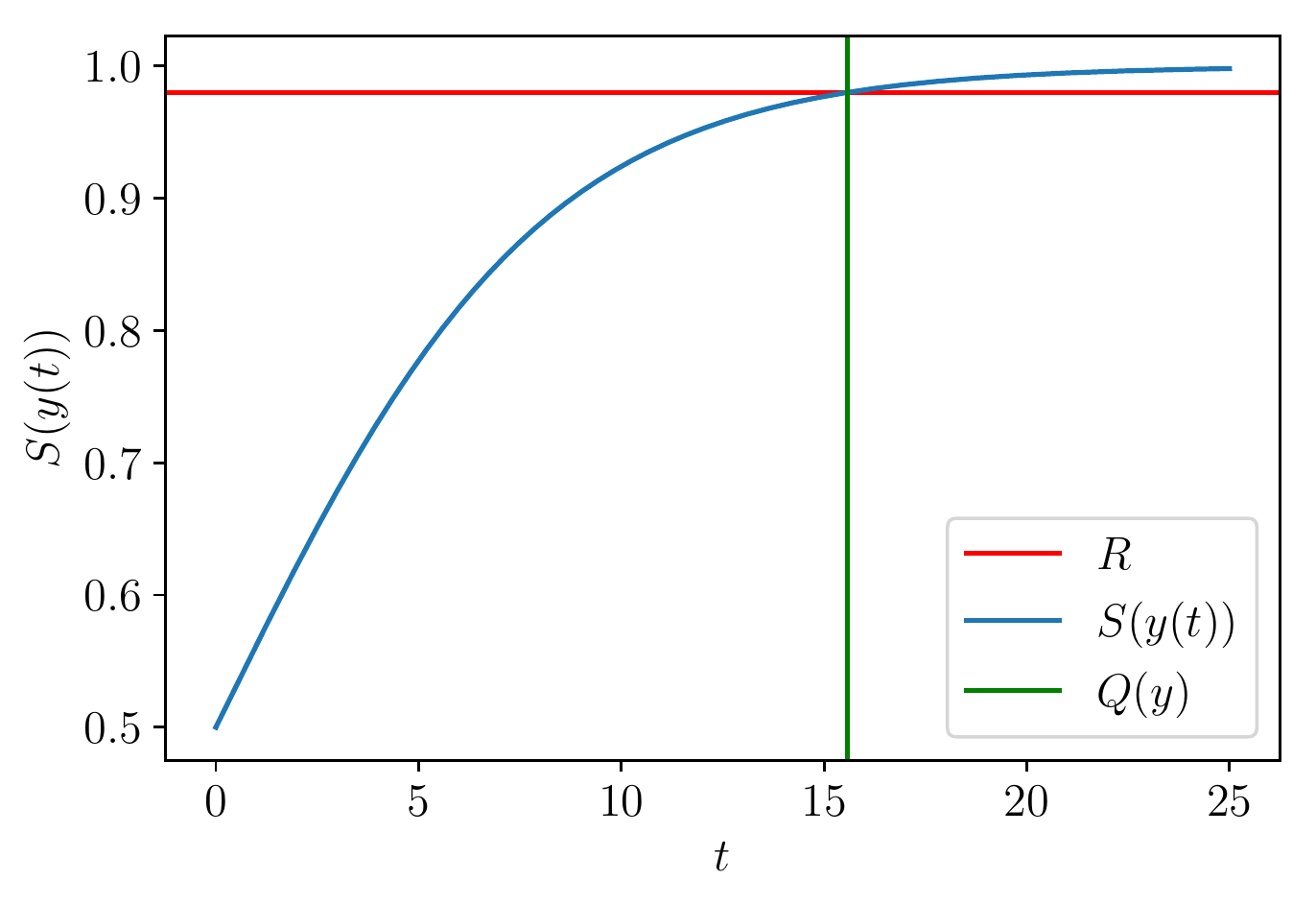}
    \caption{True values of functional and QoI for example in \S \ref{sec:example_log}, when $R=0.94$}
    \label{fig:log_true}
\end{figure}

\subsection{Conclusions for deterministic examples}
\label{sec:determ_conclusion}
Both the Taylor series and the root-finding approaches provide accurate error estimates in most cases. Some limitations of these methods were revealed in \S \ref{sec:example_harmonic_interval} and \S \ref{sec:example_harmonic_R}. The poor results in \S \ref{sec:example_harmonic_interval} are caused by the use of a low accuracy solution and the fact that computed QoI was closer to the second time the threshold value was crossed than the first. In section \ref{sec:example_harmonic_R}, specifically Tables \ref{tab:lim_ex1_1}, \ref{tab:lim_ex1_2} and \ref{tab:lim_ex1_3}, we observed that the issue that arose in \S \ref{sec:example_harmonic_interval} can be remedied by using a numerical solution that is more accurate near the QoI. Although another issue is revealed in the final column of Table \ref{tab:lim_ex1_3}, where the Taylor series approach gives poor results even though the numerical solution is quite accurate. In that instance the poor result is due to the assumption that terms involving the second derivative of $S(t)$ with respect to $t$ can be neglected. The example in \S \ref{sec:example_log} shows that the Taylor series approach may not be accurate if the error in the QoI is large, but the iterative methods are accurate provided the root finding technique finds the correct root.

\section{Numerical examples for error in the CDF of the non-standard QoI}
\label{sec:uncertainty}

The techniques outlined in \S \ref{sec:error_cdf} are applied to some examples below. The error bound \eqref{CDFerr} relies on accurate error estimates for the non-standard QoI. In the numerical examples, the estimates for the error in each sample value had an effectivity ratio close to one.

\subsection{Harmonic oscillator}
\label{sec:UQ_harmonic}

Reconsider the harmonic oscillator from \S \ref{sec:example_harmonic} this time with parameters $k$ and $m$ as random variables:
\begin{equation*}
\begin{pmatrix}
\dot{y_1}(t) \\  \dot{y_2}(t)
\end{pmatrix}
+
\begin{pmatrix}
0    & -1 \\
k/m  & 1/m
\end{pmatrix}
\begin{pmatrix}
y_1(t) \\  y_2(t)
\end{pmatrix}
=
\begin{pmatrix}
0 \\  50/m * \cos(10 t)
\end{pmatrix},
\ t \in (0,2],
\end{equation*}
with initial conditions $(y_1(0),y_2(0))=(5,0)$. Let $k$ have a normal distribution with mean 50 and \textcolor{black}{a standard deviation of} 5 and $m$ \textcolor{black}{be uniformly distributed over $[.125,.325]$}. For the QoI, set choose $R=-1$ and $S(y(t))=y_1(t)$. With $\varepsilon =\textcolor{black}{0.05}$ in \eqref{CDFerr}, the nominal CDF \eqref{nomCDF} is computed using the true solution given in \cite{Barger} with $1000$ samples. The numerical solution is obtained using cG(1) with 40 elements and the approximate CDF \eqref{approxCDF} is computed with $N=100$ samples. \textcolor{black}{Both sources contribute to the error, with the sampling error being slightly more dominant}. The computed bound is indeed larger than the actual error in the distribution. Both the bound and the error peak near the inflection point of the CDF, with the error bound being about \textcolor{black}{six} times larger than the true error.

\begin{figure}[h!]
    \centering
    \subfloat[Nominal CDF using 1000 samples and computed CDF using 100 samples for example in \S \ref{sec:UQ_harmonic}.]{\includegraphics[width=7cm]{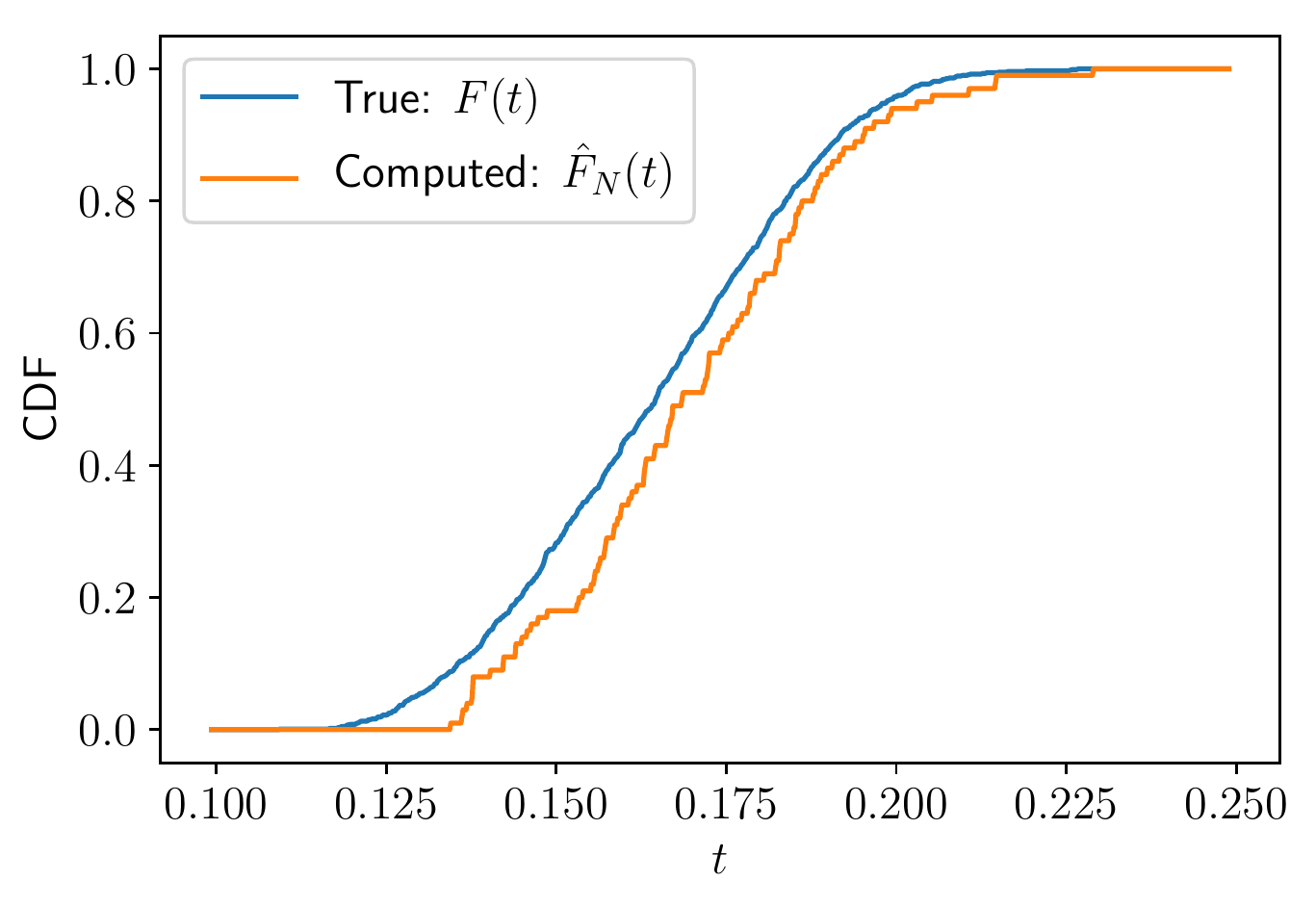} \label{fig:oscCDF}}
    \hfill
    \subfloat[Comparing nominal CDF using 1000 samples to computed CDF using 80 samples for example in \S \ref{sec:UQLorenz}.]{\includegraphics[width=7cm]{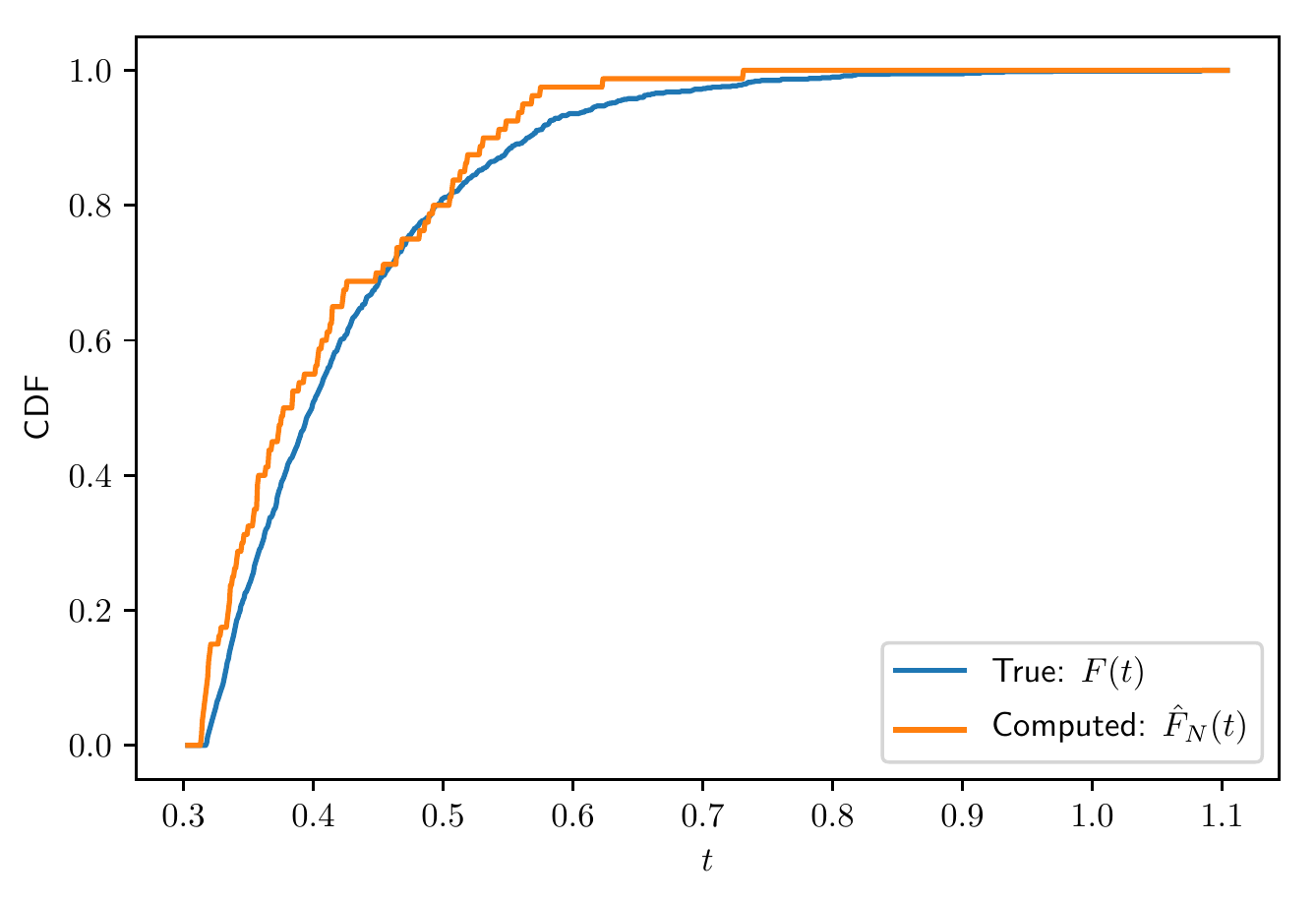} \label{fig:LorenzCDF}}

    \caption{ }

\end{figure}

\begin{figure}[h!]
    \centering
    \subfloat[Comparing computed error bound \eqref{CDFerr} to true error for the problem in \S \ref{sec:UQ_harmonic} when using 1000 samples for the nominal CDF and 100 samples for the numerical CDF.]{\includegraphics[width=7cm]{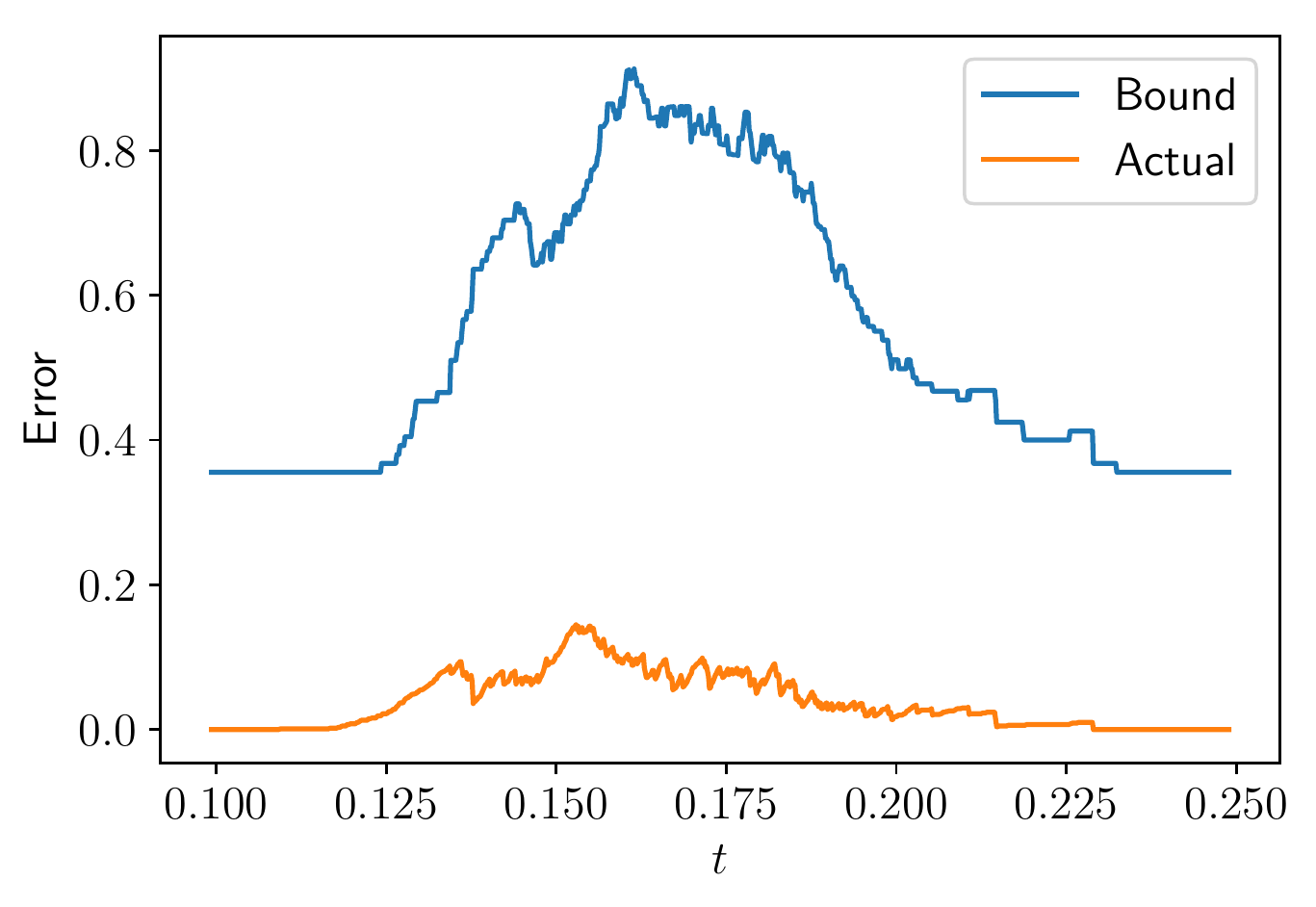}}
    \hfill
    \subfloat[Breaking the error bound into sampling and discretization contributions for the problem in \S \ref{sec:UQ_harmonic} when using 100 samples. The sampling and discretization contributions are computed as the first and second terms of \eqref{CDFerr}, respectively.]{\includegraphics[width=7cm]{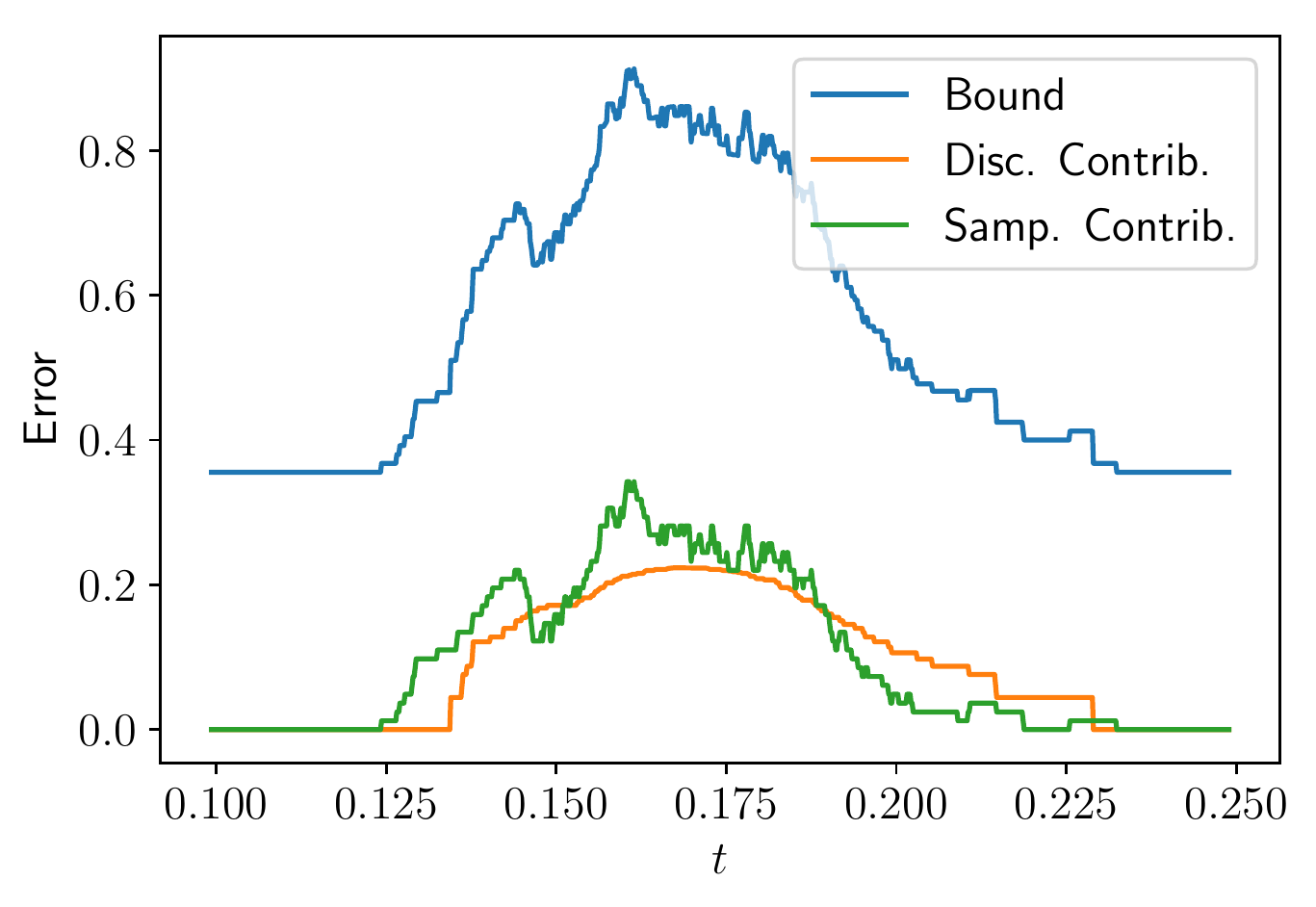}}
    \caption{Error bound for example in \S \ref{sec:UQ_harmonic}.}
    \label{fig:oscbounds}
\end{figure}

\subsection{Lorenz System}
\label{sec:UQLorenz}

Consider the Lorenz system \eqref{eq:Lorenz}, \textcolor{black}{where we let one of the initial conditions be a random variable. More precisely, $y_1(0)=\theta$ is uniformly distributed over the interval $(0,2]$.}
Again let $\sigma = 10$, $r=28$, and $b=\frac{8}{3}$. For the QoI \eqref{eq:QoI}, set $R=3$ and $S(y(t))=y_1(t)$. A reference solution and QoI are obtained using an accurate time-integrator (SciPy's solve\_ivp with event tracker) with an absolute tolerance of $10^{-15}$ and a relative tolerance of $10^{-8}$. This time, the numerical solution is computed using the cG(1) method with 30 elements.

The values needed for equation \eqref{eq:taylor_eta_approx} are
\begin{equation*}
v=(1,0,0)^\top, \; f(y,t)=(\sigma(y_2-y_1), \ r y_1 - y_2 - y_1 y_3, \ y_1 y_2 - b y_3)^\top,
\end{equation*}
and
\begin{equation*}
\nabla_y f(y,t) =
\begin{pmatrix}
    \sigma & -\sigma & 0 \\
    r-y_3 & -1 & -y_1 \\
    -y_2 & y_1 & -b \\
\end{pmatrix}.
\end{equation*}
hence, for \eqref{eq:taylor_numerator_estimate}, \eqref{eq:taylor_denominator_estimate}, and \eqref{eq:iterative_est} the data are
\begin{equation*}
\psi_1= (-1,0,0)^\top,
\psi_2= (-\sigma,\sigma,0)^\top,
\psi_3= (1,0,0)^\top.
\end{equation*}
The bound \eqref{CDFerr} is computed with $\varepsilon=0.05$. The Figure \ref{fig:LorenzCDF} compares the numerical CDF computed using 80 samples to the nominal CDF using 1000 samples. Figure \ref{fig:LorenzCDFerr} shows the discretization and sampling contributions to the calculated error bound. For this example, the discretization is the larger contributor to the error in the CDF, which is likely due to the chaotic nature of the system. As in \S \ref{sec:UQ_harmonic}  the error bound is roughly \textcolor{black}{six} times the true error at its peak.

\begin{figure}[h!]
    \centering
    \subfloat[Comparing error bound \eqref{CDFerr} to true error for example in \S \ref{sec:UQLorenz} when using 1000 samples for the nominal CDF and 80 samples for the numerical CDF.]{\includegraphics[width=7cm]{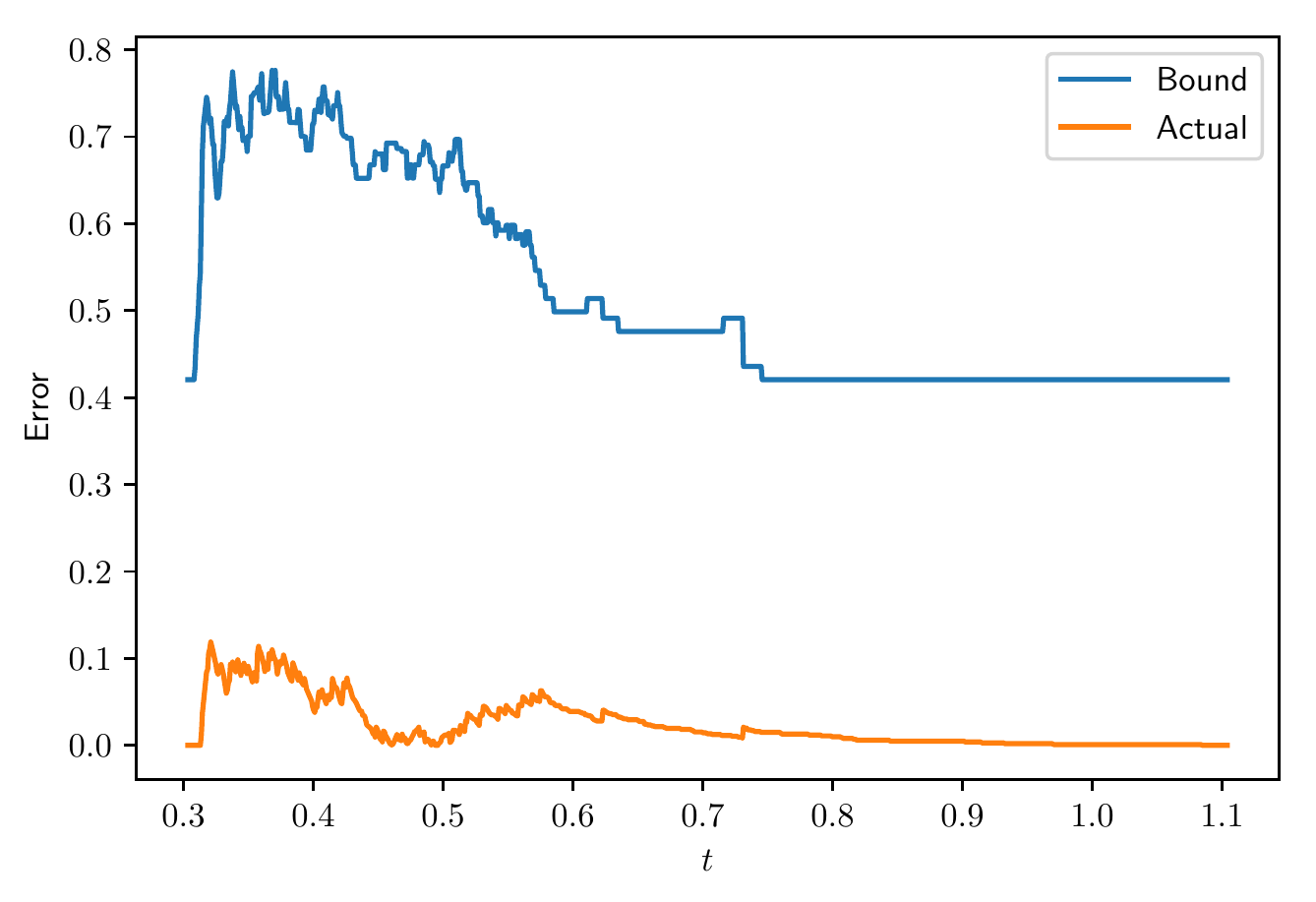}}
    \hfill
    \subfloat[Showing the sampling and discretization contributions to the error bound for the example in \S \ref{sec:UQLorenz} when using 100 samples. The sampling contribution is computed as the first term of \eqref{CDFerr}, while the second term gives the discretization contribution. ]{\includegraphics[width=7cm]{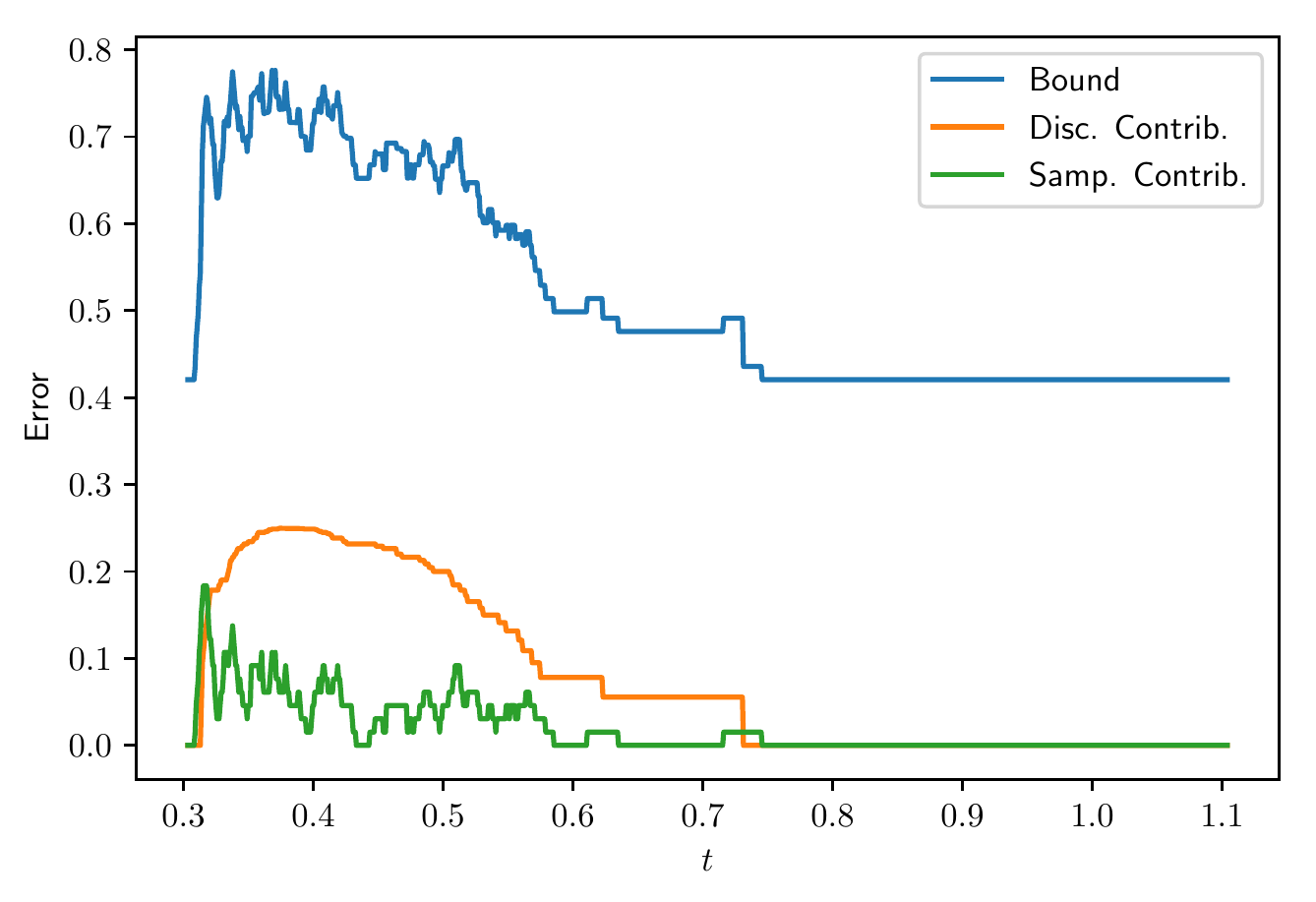}}
    \caption{}
    \label{fig:LorenzCDFerr}
\end{figure}

\section{Conclusions}
\label{sec:conclusions}

We develop two different classes of accurate \emph{a posteriori} error estimates for a QoI that cannot be expressed as a bounded functional of the solution, namely the first time when a given functional $S$ of the solution achieves a specific value. The first method is based on Taylor's Theorem and is accurate whenever the numerical solution is sufficiently accurate and the curvature of the functional $S$ is not too large. Moreover this method is  cost effective, requiring the solution of only two adjoint problems. The second class of methods are based on standard root-finding techniques and are accurate provided the numerical solution is sufficiently accurate near the event of interest. These estimates however are more costly, requiring an adjoint solution per iteration of the root-finding algorithm. Both methods can be used as a basis for determining the discretization contribution to an error bound on a CDF of the functional when one or more of the parameters governing the system of differential equations \textcolor{black}{are} random variables.

\section*{Acknowledgments}
J. Chaudhry’s  and Z. Stevens's work is supported by the NSF-DMS 1720402.
S. Tavener's and D. Estep's work is supported by NSF-DMS 1720473.

\bibliography{ns_qoi}

\begin{thebibliography}{10}

\bibitem{AO2000}
M.~Ainsworth and T.~Oden.
\newblock {\em \emph{A posteriori} error estimation in finite element
  analysis}.
\newblock John Wiley-Teubner, 2000.

\bibitem{Apostol}
Tom~M. Apostol.
\newblock {\em Calculus}, volume~1.
\newblock Wiley, 2nd edition, 1967.

\bibitem{Apostol2}
Tom~M. Apostol.
\newblock {\em Calculus}, volume~2.
\newblock Wiley, 2nd edition, 1969.

\bibitem{Ban}
W.~Bangerth and R.~Rannacher.
\newblock {\em Adaptive Finite Element Methods for Differential Equations}.
\newblock Birkhauser Verlag, 2003.

\bibitem{Barger}
V.~Barger and M.~Olsson.
\newblock {\em Classical Mechanics, A Modern Perspective}.
\newblock McGraw-Hill, New York, 1973.

\bibitem{barth04}
T.~J. Barth.
\newblock {\em \emph{A posteriori} Error Estimation and Mesh Adaptivity for
  Finite Volume and Finite Element Methods}, volume~41 of {\em Lecture Notes in
  Computational Science and Engineering}.
\newblock Springer, New York, 2004.

\bibitem{beckerrannacher}
R.~Becker and R.~Rannacher.
\newblock An optimal control approach to \emph{a posteriori} error estimation
  in finite element methods.
\newblock {\em Acta Numerica}, pages 1--102, 2001.

\bibitem{bouchard2017first}
Bruno Bouchard, Stefan Geiss, Emmanuel Gobet, et~al.
\newblock First time to exit of a continuous it{\^o} process: General moment
  estimates and $\mathrm {L}_1$-convergence rate for discrete time
  approximations.
\newblock {\em Bernoulli}, 23(3):1631--1662, 2017.

\bibitem{Cao2004}
Yang Cao and Linda Petzold.
\newblock \emph{A posteriori} error estimation and global error control for
  ordinary differential equations by the adjoint method.
\newblock {\em SIAM Journal on Scientific Computing}, 26(2):359--374, 2004.

\bibitem{CaEsTa2008}
V.~Carey, D.~Estep, and S.~Tavener.
\newblock \emph{A posteriori} analysis and adaptive error control for
  multiscale operator decomposition solution of elliptic systems {I:
  Triangular} systems.
\newblock {\em SIAM Journal on Numerical Analysis}, 47(1):740--761, 2008.

\bibitem{CEGT13b}
J.~H. Chaudhry, D.~Estep, V.~Ginting, and S.J. Tavener.
\newblock \emph{A posteriori} analysis for iterative solvers for nonautonomous
  evolution problems.
\newblock {\em SIAM/ASA J. Uncertainty Quantification}, 3(1):434--459, 2015.

\bibitem{Chaudhry2019b}
Jehanzeb Chaudhry, Don Estep, and Simon Tavener.
\newblock \emph{A posteriori} error analysis for {Schwarz} overlapping domain
  decomposition methods.
\newblock {\em arXiv e-prints}, page arXiv:1907.01139, Jul 2019.

\bibitem{Chaudhry17}
Jehanzeb~H. Chaudhry.
\newblock \emph{A posteriori} analysis and efficient refinement strategies for
  the {Poisson--Boltzmann} equation.
\newblock {\em SIAM Journal on Scientific Computing}, 40(4):A2519----A2542,
  2018.

\bibitem{ZebUQ}
Jehanzeb~H. Chaudhry, Nathanial Burch, and Donald Estep.
\newblock Efficient distribution estimation and uncertainty quantification for
  elliptic problems on domains with stochastic boundaries, 2018.

\bibitem{CCS2017}
Jehanzeb~H. Chaudhry, J.B. Collins, and John~N. Shadid.
\newblock \emph{A posteriori} error estimation for multi-stage {Runge-Kutta
  IMEX} schemes.
\newblock {\em Applied Numerical Mathematics}, 117:36--49, Jul 2017.

\bibitem{CEG+2015}
Jehanzeb~H. Chaudhry, Donald Estep, Victor Ginting, John~N. Shadid, and Simon
  Tavener.
\newblock \emph{A posteriori} error analysis of {IMEX} multi-step time
  integration methods for advection--diffusion--reaction equations.
\newblock {\em Computer Methods in Applied Mechanics and Engineering},
  285:730--751, 2015.

\bibitem{CET+2016}
Jehanzeb~Hameed Chaudhry, Don Estep, Simon Tavener, Varis Carey, and Jeff
  Sandelin.
\newblock \emph{A posteriori} error analysis of two-stage computation methods
  with application to efficient discretization and the {P}arareal algorithm.
\newblock {\em SIAM Journal on Numerical Analysis}, 54(5):2974--3002, 2016.

\bibitem{CEG+2013}
J.H. Chaudhry, D.~Estep, V.~Ginting, and S.J. Tavener.
\newblock \emph{A posteriori} analysis of an iterative multi-discretization
  method for reaction-diffusion systems.
\newblock {\em Computer Methods in Applied Mechanics and Engineering},
  267:1--22, 2013.

\bibitem{Chaudhry2019}
J.H. Chaudhry, J.N. Shadid, and T.~Wildey.
\newblock \emph{A posteriori} analysis of an {IMEX} entropy-viscosity
  formulation for hyperbolic conservation laws with dissipation.

\bibitem{iternon}
J.~H. Chaudry, D.~Estep, V.~Ginting, and S.~Tavener.
\newblock \emph{A posteriori} analysis for iterative solvers for non-autonomous
  evolution problems.
\newblock {\em SIAM Journal on Uncertainty Quantification}, 3, 2015.

\bibitem{CCH-2015}
K.A. Cliffe, J.~Collis, and P.~Houston.
\newblock Goal-oriented \emph{a posteriori} error estimation for the travel
  time functional in porous media flows.
\newblock {\em SIAM Journal of Scientific Computing}, 37(2):B127--B152, 2015.

\bibitem{CET2015}
J.~B. Collins, D.~Estep, and S.~Tavener.
\newblock \emph{A posteriori} error analysis for finite element methods with
  projection operators as applied to explicit time integration techniques.
\newblock {\em BIT Numerical Mathematics}, 55(4):1017--1042, 2015.

\bibitem{CET2014}
James~B. Collins, Don Estep, and Simon Tavener.
\newblock \emph{A posteriori} error estimation for the {Lax--Wendroff} finite
  difference scheme.
\newblock {\em Journal of Computational and Applied Mathematics}, 263:299--311,
  2014.

\bibitem{dht_mathcomp_81}
M.~Delfour, W.~Hager, and F.~Trochu.
\newblock Discontinuous {G}alerkin methods for ordinary differential equations.
\newblock {\em Math. Comp.}, 36(154):455--473, 1981.

\bibitem{dd_mathcomp_86}
M.~C. Delfour and F.~Dubeau.
\newblock Discontinuous polynomial approximations in the theory of one-step,
  hybrid and multistep methods for nonlinear ordinary differential equations.
\newblock {\em Math. Comp.}, 47(175):169--189, S1--S8, 1986.

\bibitem{dzougoutov2005adaptive}
Anna Dzougoutov, Kyoung-Sook Moon, Erik von Schwerin, Anders Szepessy, and
  Ra{\'u}l Tempone.
\newblock Adaptive monte carlo algorithms for stopped diffusion.
\newblock In {\em Multiscale methods in science and engineering}, pages 59--88.
  Springer, 2005.

\bibitem{Epper}
James~F. Epperson.
\newblock {\em An Introduction to Numerical Methods and Analysis}.
\newblock Wiley-Interscience, 2007.

\bibitem{eehj_actanum_95}
K.~Eriksson, D.~Estep, P.~Hansbo, and C.~Johnson.
\newblock Introduction to adaptive methods for differential equations.
\newblock In {\em Acta Numerica, 1995}, Acta Numerica, pages 105--158.
  Cambridge Univ. Press, Cambridge, 1995.

\bibitem{EJL04}
K.~Eriksson, C.~Johnson, and A.~Logg.
\newblock Explicit time-stepping for stiff {ODE}s.
\newblock {\em SIAM Journal on Scientific Computing}, 25(4):1142--1157, 2004.

\bibitem{estep_sinum_95}
D.~Estep.
\newblock \emph{A posteriori} error bounds and global error control for
  approximation of ordinary differential equations.
\newblock {\em SIAM J. Numer. Anal.}, 32(1):1--48, 1995.

\bibitem{Estep}
D.~Estep.
\newblock A short course on duality, adjoint operators, {Green's} functions,
  and \emph{A Posteriori} error analysis.
\newblock Unpublished, 2004.

\bibitem{Estep2009}
D.~Estep.
\newblock Error estimates for multiscale operator decomposition for
  multiphysics models.
\newblock In J.~Fish, editor, {\em Multiscale methods: bridging the scales in
  science and engineering}. Oxford University Press, USA, 2009.

\bibitem{EGT2012}
D.~Estep, V.~Ginting, and S.~Tavener.
\newblock \emph{A posteriori} analysis of a multirate numerical method for
  ordinary differential equations.
\newblock {\em Computer Methods in Applied Mechanics and Engineering},
  223:10--27, 2012.

\bibitem{DEstep}
D.~Estep, M.~Holst, and D.~Mikulencak.
\newblock Accounting for stability: \emph{a posteriori} error estimates based
  on residuals and variational analysis.
\newblock {\em Comm. Numer. Methods Engrg.}, 18:15--30, 2002.

\bibitem{ELM}
D.~Estep, M.~Larson, and R.~Williams.
\newblock Estimating the error of numerical solutions of systems of
  reaction-diffusion equations.
\newblock {\em Memoirs of the American Mathematical Society}, 696, 07 2000.

\bibitem{EstepUQ1}
D.~Estep, A.~M{\aa}lqvist, and S.~Tavener.
\newblock Nonparametric density estimation for randomly perturbed elliptic
  problems {I}: {C}omputational methods, \emph{a posteriori} analysis, and
  adaptive error control.
\newblock {\em SIAM Journal on Scientific Computing}, 31(4):2935--2959, 2009.

\bibitem{EstepUQ2}
D.~Estep, A.~M{\aa}lqvist, and S.~Tavener.
\newblock Nonparametric density estimation for randomly perturbed elliptic
  problems {II}: {A}pplications and adaptive modeling.
\newblock {\em International Journal for Numerical Methods in Engineering},
  80(6‐7), 2009.

\bibitem{Gautschi}
Walter Gautschi.
\newblock {\em Numerical Analysis}.
\newblock Birkh{\"a}user, 2011.

\bibitem{gobet2001euler}
Emmanuel Gobet.
\newblock Euler schemes and half-space approximation for the simulation of
  diffusion in a domain.
\newblock {\em ESAIM: Probability and Statistics}, 5:261--297, 2001.

\bibitem{JCC+2015}
A.~Johansson, J.H. Chaudhry, V.~Carey, D.~Estep, V.~Ginting, M.~Larson, and
  S.J. Tavener.
\newblock Adaptive finite element solution of multiscale
  {PDE}{\textendash}{ODE} systems.
\newblock {\em Computer Methods in Applied Mechanics and Engineering},
  287:150--171, 2015.

\bibitem{Logg2004}
Anders Logg.
\newblock Multi-adaptive time integration.
\newblock {\em Appl. Numer. Math.}, 48(3-4):339--354, mar 2004.

\bibitem{Rudin}
Walter Rudin.
\newblock {\em Principles of mathematical analysis}.
\newblock McGraw-Hill New York, 3d ed. edition, 1976.

\bibitem{Serfling}
Robert~J. Serfling.
\newblock {\em Approximation Theorems of Mathematical Statistics}.
\newblock John Wiley and Sons, 1980.

\end{thebibliography}
\bibliographystyle{plain}

\appendix

\section{Theorem \ref{thm:Crank-Nicolson}}
\label{sec:proof_C-N}
\begin{theorem}
\label{thm:Crank-Nicolson}
Numerical solutions obtained via the Crank-Nicolson finite difference scheme are nodally equivalent to solutions obtained using a cG(1) finite element method in which the integrals are evaluated with the trapezoidal rule.
\end{theorem}
\begin{proof}
The cG(1) formulation over a sub-interval $(t_n,t_{n+1})$, with the constant test function $v(t) = 1$ is
\begin{equation}\label{integral}
        \int_{t_n}^{t_{n+1}}\dot{y} \; {\rm d}t = \int_{t_n}^{t_{n+1}}f(y,t) \; {\rm d}t.
\end{equation}
Where by the fundamental theorem of calculus
\begin{equation}\label{left}
        \int_{t_n}^{t_{n+1}}\dot{y} \; {\rm d}t = y(t_{n+1})-y(t_n).
\end{equation}
Using the trapezoidal quadrature rule, we obtain
\begin{equation}\label{right}
\int_{t_n}^{t_{n+1}}f(y,t) \; {\rm d}t \approx
    \frac{t_{n+1}-t_n}{2}(f(y(t_{n+1}),t_{n+1})+f(y(t_n),t_n)).
\end{equation}
Substituting \eqref{left} and \eqref{right} into \eqref{integral} results in the Crank-Nicolson scheme.
\end{proof}

\end{document}